\documentclass[12pt,dvipdfmx,a4paper,reqno]{amsart}

\usepackage{amssymb}
\usepackage{amsthm}

\usepackage{hyperref}

\usepackage{amscd}

\usepackage{amsmath}
\usepackage{amsfonts}
\usepackage{latexsym}
\usepackage{graphicx}
\usepackage{enumerate}
\usepackage{color}
\usepackage{comment}

\usepackage[english]{babel}
\usepackage[latin1]{inputenc}
\newtheorem{theorem}{Theorem}[section]
\newtheorem{lemma}[theorem]{Lemma}
\newtheorem{proposition}[theorem]{Proposition}
\newtheorem{corollary}[theorem]{Corollary}
\newtheorem{remark}[theorem]{Remark}

\newtheorem{definition}[theorem]{Definition}

\makeatletter
 \@addtoreset{equation}{section}
\makeatother
\def\theequation{\arabic{section}.\arabic{equation}}

\def\d{\mathrm{d}}
\def\n{\mathrm{n}}

\def\Cset{\mathbb{C}}
\def\Nset{\mathbb{N}}
\def\Rset{\mathbb{R}}
\def\Qset{\mathbb{Q}}

\def\Zset{\mathbb{Z}}

\title[Nonintegrability of the unfoldings of codimension-two bifurcations]
{Nonintegrability of the unfoldings of codimension-two bifurcations}

\author[P.~Acosta-Hum\'anez]{Primitivo B. Acosta-Hum\'anez}
\address[P.~Acosta-Hum\'anez]{Instituto Superior de Formaci\'on Docente Salom\'e Ure\~na, Recinto Emilio Prud'Homme, Santiago de los Caballeros, Dominican Republic \& School of Basic and Biomedical Sciences, Universidad Sim\'on Bol\'{\i}var, Barranquilla - Colombia.}
\email{primitivo.acosta-humanez@isfodosu.edu.do}

\author[K. Yagasaki]{Kazuyuki Yagasaki}
\address[K. Yagasaki]{Department of Applied Mathematics and Physics, Kyoto University,
 Yoshida-Honmachi, Sakyo-ku, Kyoto 606-8501, Japan} %
\email{yagasaki@amp.i.kyoto-u.ac.jp}

\begin{document}

\date{}



\begin{abstract}
Codimension-two bifurcations are fundamental and interesting phenomena in dynamical systems.
Fold-Hopf and double-Hopf bifurcations are the most important among them.
We study the unfoldings of these two codimension-two bifurcations,
 and obtain sufficient conditions for their nonintegrability
 in the meaning of Bogoyavlenskij.
We reduce the problems of the unfoldings to those of planar polynomial vector fields
 and analyze the nonintegrability of the planar vector fields,
 based on Ayoul and Zung's version of the Morales-Ramis theory.
New useful criteria for nonintegrability of planar polynomial vector fields are also obtained.
The approaches used here are applicable to many problems including circular symmetric systems.\end{abstract}

\maketitle

\section{Introduction}

Codimension-two bifurcations are fundamental and interesting phenomena in dynamical systems
 and have been studied extensively
 since the seminal papers of Arnold \cite{A72} and Takens \cite{T74}.
Fold-Hopf and double-Hopf bifurcations are the most important among them,
 and now well described in several textbooks such as \cite{GH,Ku}.
For the former, fold (saddle-node) and Hopf bifurcation curves
 meet at the bifurcation point and its unfolding (or normal form) is given by
\begin{equation}
\begin{split}
&\dot{x}_1=\nu x_1-\omega x_2+\alpha x_1x_3-\beta x_2x_3,\\
&\dot{x}_2=\omega x_1+\nu x_2+\beta x_1x_3+\alpha x_2x_3,\\
&\dot{x}_3=\mu+s(x_1^2+x_2^2)+x_3^2,
\end{split}
\qquad
x=(x_1,x_2,x_3)\in\Rset^3,
\label{eqn:fH}
\end{equation}
where $\mu,\nu\neq 0$, $\omega>0$, $\alpha,\beta\in\Rset$,
 $s=\pm 1$
 and the dot represents differentiation with respect to the independent variable $t$.
For the latter, two Hopf bifurcation curves meet at the bifurcation point
 and its unfolding is given by
\begin{equation}
\begin{split}
\dot{x}_1
=& -\omega_1 x_2+(\nu+s(x_1^2+x_2^2)+\alpha(x_3^2+x_4^2))x_1,\\
\dot{x}_2
=& \omega_1 x_1+(\nu+s(x_1^2+x_2^2)+\alpha(x_3^2+x_4^2))x_2,\\
\dot{x}_3
=& -\omega_2 x_4+(\mu+\beta(x_1^2+x_2^2)-(x_3^2+x_4^2))x_3,\\
\dot{x}_4
=& \omega_2 x_3+(\mu+\beta(x_1^2+x_2^2)-(x_3^2+x_4^2))x_4,
\end{split}
\quad
x=(x_1,x_2,x_3,x_4)\in\Rset^4,
\label{eqn:dH}
\end{equation}
where $\mu,\nu\neq 0$, $\omega>0$, $\alpha,\beta\in\Rset$ and $s=\pm 1$.
The unfoldings \eqref{eqn:fH} and \eqref{eqn:dH} are universal,
 i.e., their bifurcation diagrams do not qualitatively change
 near the bifurcation points
 even if higher-order terms are included, in some cases,
 but they are not universal and may exhibit complicated dynamics such as chaos
 if higher-order terms are included, in the other cases.
See \cite{DIKS13,Ku} for more details.
 
Recently, in \cite{Y18},
 the nonintegrability of the unfolding \eqref{eqn:fH} for fold-Hopf bifurcations
 was shown for almost all parameter values of $\omega$ and $\alpha,\beta\in\Rset$
 when $\mu,\nu\neq 0$.
More precisely the following theorem was proved.

\begin{theorem}
\label{thm:pre}
Let $\mu,\nu,\alpha,\beta,\omega\in\Cset$.
Suppose that $\mu,\nu\neq 0$, $\alpha\pm\nu/\sqrt{-\mu}\not\in\Qset$
 and $2\alpha\not\in\Zset_{\le 0}:=\{k\in\Zset\mid k\le 0\}$.
Then the complexification of \eqref{eqn:fH} with $s=\pm 1$
 is meromorphically nonintegrable
 near the $x_3$-plane %$\{(0,0,x_3)\mid x_3\in\Cset\}$
 in $\Cset^3$.
\end{theorem}
Here the following definition of integrability due to Bogoyavlenskij \cite{B} has been adopted.
\begin{definition}[Bogoyavlenskij]
\label{dfn:1}
Consider systems %of differential equations of the form
\begin{equation}%\label{DynSystem}
\dot{x} = v(x),\quad x\in D\subset\Cset^n,
\label{eqn:gsys}
\end{equation}
where $n>0$ is an integer, $D$ is a region in $\Cset^n$ and $v:D\to \Cset^n$ is holomorphic.
Let $q$ be an integer such that $1\le q\le n$.
Eq.~\eqref{eqn:gsys} is called $(q,n-q)$-\emph{integrable} or simply \emph{integrable} 
 if there exist $q$ vector fields $v_1(x)(:= v(x)),v_2(x),\dots,v_q(x)$
 and $n-q$ scalar-valued functions $F_1(x),\dots,F_{n-q}(x)$
 such that the following two conditions hold:
\begin{enumerate}
\item[\rm(i)]
$v_1,\dots,v_q$ are linearly independent almost everywhere
and commute with each other, i.e.,
\[
[v_j, v_k]:= \frac{\partial v_k}{\partial x}v_j-\frac{\partial v_j}{\partial x}v_k=0
\]
for $j,k=1,\ldots,q$;
\item[\rm(ii)]
$\partial F_1/\partial x,\dots,\partial F_{n-q}/\partial x$
 are linearly independent almost everywhere
 and $F_1,\dots,$\linebreak$F_{n-q}$ are first integrals of $v_1, \dots,v_q$, i.e.,
\[
\frac{\partial F_k}{\partial x}v_j=0\quad\mbox{for $j=1,\ldots,q$ and $k=1,\ldots,n-q$}.
\]
\end{enumerate}
If $v_1,v_2,\dots,v_q$ and $F_1,\dots,F_{n-q}$ are meromorphic and rational, respectively,
 then Eq.~\eqref{eqn:gsys} is said
 to be \emph{meromorphically} and \emph{rationally integrable}.
\end{definition}
%\textcolor{red}{Along this paper we use the notation $$\frac{\partial f}{\partial x_k}(x_1,x_2,\ldots, x_n)$$ instead of the notation $$\frac{\partial f(x_1,x_2,\ldots, x_n)}{\partial x_k}.$$ }
% KY: This is a standard notation.
Definition~\ref{dfn:1} is regarded as a generalization of the Liouville integrability
 for Hamiltonian systems
 since if a Hamiltonian system with $n$ degrees of freedom is Liouville integrable,
 then there exist $n$ functionally independent first integrals
 and $n$ linearly independent vector fields corresponding to the first integrals (almost everywhere).
The statement similar to that of the Liouville-Arnold theorem \cite{A89}
 also holds for integrable systems in the meaning of Bogoyavlenskij:
 if Eq.~\eqref{eqn:gsys} is integrable
 and the level set $F^{-1}(c)$ with $F(x):= (F_1(x),\ldots,F_{n-q}(x))$ is compact for $c\in\Cset^{n-q}$,
 then it can be transformed to linear flow on the $q$-dimensional torus $\mathbb{T}^q$. 
See \cite{B} for more details.

For general Hamiltonian systems,
 Morales-Ruiz and Ramis \cite{MR} developed a strong method
 to present a sufficient condition for their meromorphic or rational nonintegrability.
Their theory, which is now called the Morales-Ramis theory,
 states that complex Hamiltonian systems are meromorphically or rationally nonintegrable
 if the identity components of the differential Galois groups \cite{CH11,VS03}
 for their variational equations (VEs) or normal variational equations (NVEs)
 around particular nonconstant solutions such as periodic orbits are not commutative.
Moreover, the Morales-Ramis theory was extended in \cite{MRS},
 so that weaker sufficient conditions for nonintegrability can be obtained
 by using higher-order VEs or NVEs.
See also \cite{M}.
Furthermore, Ayoul and Zung \cite{AZ}
 showed that the Morales-Ramis theory is also applicable
 for detection of meromorphic or rational nonintegrability of non-Hamiltonian systems
 in the meaning of Bogoyavlenskij.
For the proof of Theorem~\ref{thm:pre} in \cite{Y18},
 the generalization of the Morales-Ramis theory due to Ayoul and Zung was used.
The following questions were also given in \cite{Y18}:
\begin{itemize}
\item 
Is the unfolding \eqref{eqn:fH} for fold-Hopf bifurcations meromorphically nonintegrable
 when $\alpha+\nu/\sqrt{-\mu}\in\Qset$, $\alpha-\nu/\sqrt{-\mu}\in\Qset$ or $2\alpha\in\Zset_{\le 0}$?
\item
Is the unfolding \eqref{eqn:dH} of double Hopf bifurcations
 also meromorphically nonintegrable for almost all parameter values like \eqref{eqn:fH}?
\end{itemize}

In this paper, we study the nonintegrability
 of the unfoldings \eqref{eqn:fH} and \eqref{eqn:dH}
 for the fold-Hopf and double-Hopf bifurcations, respectively,
 in the meaning of Bogoyavlenskij,
 and give sufficient conditions for their nonintegrability.
Our main results are precisely stated as follows.

\begin{theorem}
\label{thm:fH1}
Let $\mu,\nu,\alpha,\beta,\omega\in\Cset$.
Suppose that one of the following conditions holds$:$
\begin{enumerate}
\item[\rm(i)]
$\mu\neq 0$, $\alpha\not\in\Qset$ and $\nu\neq 0;$
\item[\rm(ii)]
$\mu\neq 0$, $\nu/\sqrt{-\mu}\notin\Qset$ and $2\alpha-1\notin\Zset_{\le 0};$
\item[\rm(iii)]
$\mu=0$, $\nu\neq 0$ and $2\alpha-1\notin\Zset_{\le 0}$.
\end{enumerate}
Then the complexification of \eqref{eqn:fH} with $s=\pm 1$
 is meromorphically nonintegrable near the $x_3$-plane in $\Cset^3$.
\end{theorem}

\begin{theorem}
\label{thm:dH1}
Let $\mu,\nu,\alpha,\beta,\omega_1,\omega_2\in\Cset$.
Suppose that one of the following conditions holds$:$
\begin{enumerate}
\item[\rm(i)]
$\mu\neq 0$, $\nu/\mu\notin\Qset$, $\alpha\notin\Zset_{\ge 0}:=\{k\in\Zset\mid k\ge 0\}$,
 $\alpha+\nu/\mu+2\not\in\Zset_{\le 0}$
 and $\beta\nu-\mu s,(\alpha\mu+\nu)s-(\beta\nu-\mu s)\neq 0;$
\item[\rm(ii)]
$\mu\neq 0$, $\alpha+\nu/\mu\notin\Qset$, $\alpha\notin\Zset_{\ge 0}$
 and $\beta\nu-\mu s,(\alpha\mu+\nu)s-(\beta\nu-\mu s)\neq 0;$
\item[\rm(iii)]
$\mu=0$, $\nu\neq 0$, $\alpha\notin\Zset_{\ge 0}$ and $\beta\neq s$.
\end{enumerate}
Then the complexification of \eqref{eqn:dH} with $s=\pm 1$
 is meromorphically nonintegrable near the $(x_1,x_2)$-plane in $\Cset^4$.
\end{theorem}

\begin{theorem}
\label{thm:dH2}
Let $\mu,\nu,\alpha,\beta,\omega_1,\omega_2\in\Cset$.
Suppose that one of the following conditions holds$:$
\begin{enumerate}
\item[\rm(i)]
$\nu\neq 0$, $\mu/\nu\notin\Qset$, $\beta s\notin\Zset_{\le 0}$,
 $\beta s-\mu/\nu-2\in\Zset_{\ge 0}$
 and $\alpha\mu+\nu,\beta\nu s-\mu-(\alpha\mu+\nu)\neq 0;$
\item[\rm(ii)]
$\nu\neq 0$, $\beta s-\mu/\nu\notin\Qset$, $\beta s\notin\Zset_{\le 0}$
 and $\alpha\mu+\nu,\beta\nu s-\mu-(\alpha\mu+\nu)\neq 0;$
\item[\rm(iii)]
$\nu=0$, $\mu\neq 0$, $\beta s\notin\Zset_{\le 0}$ and $\alpha\neq-1$.
\end{enumerate}
Then the complexification of \eqref{eqn:dH} with $s=\pm 1$
 is meromorphically nonintegrable near the $(x_3,x_4)$-plane in $\Cset^4$.
\end{theorem}

Note that $\alpha\pm\nu/\sqrt{-\mu}\in\Qset$ if and only if $\alpha,\nu/\sqrt{-\mu}\in\Qset$.
In particular, for \eqref{eqn:fH},
 our sufficient condition in Theorem~\ref{thm:fH1}
 is much weaker than that of Theorem~\ref{thm:pre}
 except for $\alpha=1/2$, $\mu\neq 0$ and $\nu/\sqrt{-\mu}\notin\Qset$.
Thus, we provide (possibly partial) answers to the above questions
 raised up for \eqref{eqn:fH} and \eqref{eqn:dH} in \cite{Y18}.

Our approaches to prove the above main theorems are as follows.
We first use the change of coordinate $(x_1,x_2)=(r\cos\theta,r\sin\theta)$
 to transform \eqref{eqn:fH} to
\begin{equation}
\dot r=(\nu+\alpha x_3)r,\quad
\dot x_3=\mu+sr^2+x_3^2,\quad
\dot\theta=\omega+\beta x_3.
\label{eqn:fHp}
\end{equation}
The $(r,x_3)$-components are independent of $\theta$.
Using the change of coordinates $(x_1,x_2)=(r_1\cos\theta_1,r_1\sin\theta_1)$
 and $(x_3,x_4)=(r_2\cos\theta_2,r_2\sin\theta_2)$,
 we also transform \eqref{eqn:dH} to
\begin{equation}
%\begin{split}
%&
\dot{r}_1
= r_1(\nu+s r_1^2+\alpha r_2^2),\quad
\dot{r}_2
= r_2(\mu+\beta r_1^2-r_2^2),\quad
\dot{\theta}_1
=\omega_1,\quad
\dot{\theta}_2
=\omega_2.
%\end{split}
\label{eqn:dHp}
\end{equation}
The $(r_1,r_2)$-components are independent of $\theta_1$ and $\theta_2$.
We show that one can reduce the nonintegrability of \eqref{eqn:fH} and \eqref{eqn:dH}
 to that of the $(r,x_3)$-components of \eqref{eqn:fHp},
\begin{equation}
\dot r=(\alpha x_3+\nu)r,\quad
\dot x_3=\mu+sr^2+x_3^2,
\label{eqn:fHp0}
\end{equation}
and the $(r_1,r_2)$-components of \eqref{eqn:dHp},
\begin{equation}
\dot{r}_1
= r_1(\nu+s r_1^2+\alpha r_2^2),\quad
\dot{r}_2
= r_2(\mu+\beta r_1^2-r_2^2),
\label{eqn:dHp0}
\end{equation}
respectively.
See Corollaries~\ref{cor:fH} and \ref{cor:dH} below.

On the other hand, one of the authors and his coworkers \cite{ALMP2}
 recently proposed an approach to obtain sufficient conditions
 for nonintegrability of such planar polynomial vector fields
 based on Ayoul and Zung's version \cite{AZ} of the Morales-Ramis theory \cite{M,MR,MRS}.
Similar approaches based on the differential Galois theory
 were used earlier for linear second-order differential equations in \cite{AP}
 and special planar polynomial vector fields in \cite{ALMP1}.
We extend their discussions
 to obtain new criteria for nonintegrability of planar polynomial vector fields
 and apply them to \eqref{eqn:fHp0} and \eqref{eqn:dHp0}
 for proving Theorems~\ref{thm:fH1}-\ref{thm:dH2}.
The approaches used here are applicable to many problems
 including circular symmetric systems.

The outline of this paper is as follows.
In Section~2 we give the key result to reduce the problems of \eqref{eqn:fH} and \eqref{eqn:dH}
 to those of \eqref{eqn:fHp0} and \eqref{eqn:dHp0}, respectively.
In Section~3 we review a necessary part of Acosta-Hum\'anez et al. \cite{ALMP2}
 for nonintegrability of planar polynomial vector fields
 and extend their discussion to give the other key result
 to analyze \eqref{eqn:fHp0} and \eqref{eqn:dHp0}.
The proof of Theorem~\ref{thm:fH1} is provided in Section~4,
 and the proofs of Theorems~\ref{thm:dH1} and \ref{thm:dH2} are provided in Section~5.

\section{Reduction of the unfoldings to two-dimensional systems}

Let $m>0$ be an integer and consider $m+2$-dimensional systems of the form
\begin{equation}
\dot{x}=f(x,y),\quad
\dot{y}=g(x,y),\quad
(x,y)\in D,
\label{eqn:fg}
\end{equation}
where $D\subset\Cset^2\times\Cset^m$ is a region containing $m$-dimensional plane
 $\{(0,y)\in\Cset^2\times\Cset^m\mid y\in\Cset^m\}$,
 and $f:D\to\Cset^2$ and $g:D\to\Cset^m$ are analytic.
Assume that by the change of coordinates
 $x=(x_1,x_2)=(r\cos\theta,r\sin\theta)$,
 Eq.~\eqref{eqn:fg} is transformed to
\begin{equation}
\dot{r}=R(r,y),\quad
\dot{y}=\tilde{g}(r,y),\quad
\dot{\theta}=\Theta(r,y),\quad
(r,y,\theta)\in\tilde{D}\times\Cset,
\label{eqn:Rg0}
\end{equation}
where $\tilde{D}\subset\Cset\times\Cset^m$ is a region
 containing the $m$-dimensional $y$-plane
 $\{(0,y)\in\Cset\times\Cset^m\mid y\in\Cset^m$\},
 and $R:\tilde{D}\to\Cset$, $\tilde{g}:\tilde{D}\to\Cset^m$
 and $\Theta:\tilde{D}\to\Rset$ are analytic.
Note that $\tilde{g}(r,y)=g(r\cos\theta,r\sin\theta,y)$.
We are especially interested in the $(r,y)$-components of \eqref{eqn:Rg0},
\begin{equation}
\dot{r}=R(r,y),\quad
\dot{y}=\tilde{g}(r,y),
\label{eqn:Rg}
\end{equation}
which are independent of $\theta$.
In this situation we have the following proposition.

\begin{proposition}
\label{prop:2a}
\
\begin{enumerate}
\item[\rm(i)]
Suppose that Eq.~\eqref{eqn:fg} has a meromorphic first integral $F(x_1,x_2,y)$ near $(x_1,x_2)=(0,0)$,
 and let $\tilde{F}(r,\theta,y)=F(r\cos\theta,r\sin\theta,y)$.
If $\tilde{g}_j(0,y)\neq 0$ for almost all $y\in\tilde{D}$ for some $j=1,\ldots,m$, then
\[
G(r,y)=\tilde{F}(r,\tilde{\theta}_j(y_j),y)
\]
is a meromorphic first integral of \eqref{eqn:Rg} near $r=0$,
 where $y_j$ and $\tilde{g}_j(r,y)$ are the $j$-th components of $y$ and $\tilde{g}(r,y)$, respectively,
 and $\tilde{\theta}_j(y_j)$ represents the $\theta$-component of a solution to
\begin{equation}
\frac{\d r}{\d y_j}=\frac{R(r,y)}{\tilde{g}_j(r,y)},\quad
\frac{\d y_\ell}{\d y_j}=\frac{\tilde{g}_\ell(r,y)}{\tilde{g}_j(r,y)},\quad
\frac{\d\theta}{\d y_j}=\frac{\Theta(r,y)}{\tilde{g}_j(r,y)},\quad\ell\neq j.
\label{eqn:prop2a1}
\end{equation}
\item[\rm(ii)]
Suppose that Eq.~\eqref{eqn:fg} has a  meromorphic commutative vector field
\begin{equation}
v(x_1,x_2,y):=
\begin{pmatrix}
v_1(x_1,x_2,y)\\
v_2(x_1,x_2,y)\\
v_y(x_1,x_2,y)
\end{pmatrix}
\label{eqn:prop2a2}
\end{equation}
with $v_1,v_2:D\to\Cset$ and $v_y:D\to\Cset^m$ near $(x_1,x_2)=(0,0)$.
If $\Theta(0,y)\neq 0$ for almost all $y\in\tilde{D}$, then
\begin{equation}
\begin{pmatrix}
\tilde{v}_r(r,\theta,y)\\
\tilde{v}_y(r,\theta,y)
\end{pmatrix}
=
\begin{pmatrix}
v_1(r\cos\theta,r\sin\theta,y)\cos\theta+v_2(r\cos\theta,r\sin\theta,y)\sin\theta\\
v_y(r\cos\theta,r\sin\theta,y)
\end{pmatrix}
\label{eqn:prop2a3}
\end{equation}
is independent of $\theta$
 and it is a  meromorphic commutative vector field of \eqref{eqn:Rg} near $r=0$.
\end{enumerate}
\end{proposition}

\begin{proof}
(i)
Assume that $F(x_1,x_2,y)$ is a meromorphic first integral of \eqref{eqn:fg} near $(x_1,x_2)=(0,0)$
 and $\tilde{g}_j(0,y)\neq 0$ for almost all $y\in\tilde{D}$ for some $j=1,\ldots,m$.
Then $\tilde{F}(r,\theta,y)$ is a first integral of \eqref{eqn:Rg0}, so that
\begin{align*}
&
\frac{\partial G}{\partial r}(r,y)R(r,y)
+\frac{\partial G}{\partial y}(r,y)\tilde{g}(r,y)\\
&
=\frac{\partial}{\partial r}\tilde{F}(r,\tilde{\theta}_j(y_j),y)R(r,y)
+\frac{\partial}{\partial y}\tilde{F}(r,\tilde{\theta}_j(y_j),y)\tilde{g}(r,y)\\
&
=\frac{\partial\tilde{F}}{\partial r}(r,\tilde{\theta}_j(y_j),y)R(r,y)
+\frac{\partial\tilde{F}}{\partial y}(r,\tilde{\theta}_j(y_j),y)\tilde{g}(r,y)
+\frac{\partial\tilde{F}}{\partial\theta}(r,\tilde{\theta}_j(y_j),y)\frac{\d\tilde{\theta}_j}{\d y_j}(y_j)\tilde{g}_j(r,y)\\
&
=\frac{\partial\tilde{F}}{\partial r}(r,\tilde{\theta}_j(y_j),y)R(r,y)
+\frac{\partial\tilde{F}}{\partial y}(r,\tilde{\theta}_j(y_j),y)\tilde{g}(r,y)
+\frac{\partial\tilde{F}}{\partial\theta}(r,\tilde{\theta}_j(y_j),y)\Theta(r,y)=0.
\end{align*}
Here we have used the fact
 that $\tilde{\theta}_j(y_j)$ is the $\theta$-component of a solution to \eqref{eqn:prop2a1}.
Note that $\tilde{F}(r,\theta,y)$ is meromorphic since so is $F(x_1,x_2,y)$
 and that the solution is analytic since so are $R(r,y),\tilde{g}(R,Y),\Theta(r,y)$.
Thus, we obtain the desired result.

(ii)
Assume that Eq.~\eqref{eqn:prop2a2}
 gives a meromorphic commutative vector field of \eqref{eqn:fg} near $(x_1,x_2)=(0,0)$
 and $\Theta(0,y)\neq 0$.
Let
\[
\tilde{v}_\theta(r,\theta,y)
=-v_1(r\cos\theta,r\sin\theta,y)\frac{\sin\theta}{r}
 +v_2(r\cos\theta,r\sin\theta,y)\frac{\cos\theta}{r},
\]
which is also meromorphic.
Then
\[
\tilde{v}(r,y,\theta)=
\begin{pmatrix}
\tilde{v}_r(r,\theta,y)\\[0.5ex]
\tilde{v}_y(r,\theta,y)\\[0.5ex]
\tilde{v}_\theta(r,\theta,y)
\end{pmatrix}
\]
is also a commutative vector field of \eqref{eqn:Rg0}, i.e.,
\begin{equation}
\begin{split}
&
\frac{\partial R}{\partial r}(r,y)\tilde{v}_r(r,\theta,y)
 +\frac{\partial R}{\partial y}(r,y)\tilde{v}_y(r,\theta,y)\\
&
 -\frac{\partial\tilde{v}_r}{\partial r}(r,\theta,y)R(r,y)
 -\frac{\partial\tilde{v}_r}{\partial y}(r,\theta,y)\tilde{g}(r,y)
 -\frac{\partial\tilde{v}_r}{\partial\theta}(r,\theta,y)\Theta(r,y)=0,\\
&
\frac{\partial\tilde{g}}{\partial r}(r,y)\tilde{v}_r(r,\theta,y)
 +\frac{\partial\tilde{g}}{\partial y}(r,y)\tilde{v}_y(r,\theta,y)\\
&
 -\frac{\partial\tilde{v}_y}{\partial r}(r,\theta,y)R(r,y)
 -\frac{\partial\tilde{v}_y}{\partial y}(r,\theta,y)\tilde{g}(r,y)
  -\frac{\partial\tilde{v}_y}{\partial\theta}(r,\theta,y)\Theta(r,y)=0,\\
&
\frac{\partial\Theta}{\partial r}(r,y)\tilde{v}_r(r,\theta,y)
 +\frac{\partial\Theta}{\partial y}(r,y)\tilde{v}_y(r,\theta,y)\\
&
 -\frac{\partial\tilde{v}_\theta}{\partial r}(r,\theta,y)R(r,y)
 -\frac{\partial\tilde{v}_\theta}{\partial y}(r,\theta,y)\tilde{g}(r,y)
 -\frac{\partial\tilde{v}_\theta}{\partial\theta}(r,\theta,y)\Theta(r,y)=0.
\end{split}
\label{eqn:prop2a6}
\end{equation}
Let $(r,y,\theta)=(\bar{r}(t),\bar{y}(t),\bar{\theta}(t))$
 be a solution to \eqref{eqn:Rg0} as in the proof of part (i).
From \eqref{eqn:prop2a6} we see that
 $\chi=\tilde{v}(\bar{r}(t),\bar{y}(t),\bar{\theta}(t))$
 is a solution to the VE of \eqref{eqn:Rg0} along the solution,
\[
\dot{\chi}=
\begin{pmatrix}
\displaystyle\frac{\partial R}{\partial r}(\bar{r}(t),\bar{y}(t)) &
\displaystyle\frac{\partial R}{\partial y}(\bar{r}(t),\bar{y}(t)) &
0\\[2ex]
\displaystyle\frac{\partial\tilde{g}}{\partial r}(\bar{r}(t),\bar{y}(t)) &
\displaystyle\frac{\partial\tilde{g}}{\partial y}(\bar{r}(t),\bar{y}(t)) &
0\\[2ex]
\displaystyle\frac{\partial\Theta}{\partial r}(\bar{r}(t),\bar{y}(t)) &
\displaystyle\frac{\partial\Theta}{\partial y}(\bar{r}(t),\bar{y}(t)) &
0
\end{pmatrix}\chi.
\]
Hence,
\[
\chi=
\begin{pmatrix}
\tilde{v}_r(\bar{r}(t),\bar{y}(t),\bar{\theta}(t))\\
\tilde{v}_y(\bar{r}(t),\bar{y}(t),\bar{\theta}(t))
\end{pmatrix}
\]
is a solution to the VE of \eqref{eqn:Rg}
 along the solution $(\bar{r}(t),\bar{y}(t))$,
\[
\dot{\chi}=
\begin{pmatrix}
\displaystyle\frac{\partial R}{\partial r}(\bar{r}(t),\bar{y}(t)) &
\displaystyle\frac{\partial R}{\partial y}(\bar{r}(t),\bar{y}(t))\\[2ex]
\displaystyle\frac{\partial\tilde{g}}{\partial r}(\bar{r}(t),\bar{y}(t)) &
\displaystyle\frac{\partial\tilde{g}}{\partial y}(\bar{r}(t),\bar{y}(t))
\end{pmatrix}\chi.
\]
This means that 
\[
\frac{\partial\tilde{v}_r}{\partial\theta}(r,\theta,y)\Theta(r,y)
=\frac{\partial\tilde{v}_y}{\partial\theta}(r,\theta,y)\Theta(r,y)=0,
\]
along with \eqref{eqn:prop2a6}.
Since $\Theta(r,y)\neq 0$ for almost all $y\in\tilde{D}$ near $r=0$,
 we obtain the desired result.
\end{proof}

\begin{remark}\
\begin{enumerate}
\item[\rm(i)]
As in Proposition~{\rm\ref{prop:2a}(i)}, we can also show that
 if Eq.~\eqref{eqn:fg} has a first integral $F(x_1,x_2,y)$ near $(x_1,x_2)=(0,0)$
 and $R(0,y)\neq 0$ for almost all $y\in\tilde{D}$,
 then $G(r,y)=\tilde{F}(r,\tilde{\theta}(r),y)$ is a first integral of \eqref{eqn:Rg} near $r=0$,
 where $\tilde{\theta}(r)$ represents the $\theta$-component of a solution to
\[
\frac{\d y}{\d r}=\frac{\tilde{g}(r,y)}{R(r,y)},\quad
\frac{\d\theta}{\d r}=\frac{\Theta(r,y)}{R(r,y)}.
\]
\item[\rm(ii)]
If Eq.~\eqref{eqn:fg} has a commutative vector field $v(x_1,x_2,y)$ and $\Theta(r,y)\equiv 0$,
 then by \eqref{eqn:prop2a6} Eq.~\eqref{eqn:prop2a3} gives a commutative vector field of \eqref{eqn:Rg}
 for any $\theta\in\Cset$.
\end{enumerate}
\end{remark}

Using Proposition~\ref{prop:2a} for \eqref{eqn:fH} and \eqref{eqn:dH}
 (once for the former and twice for the latter),
 we immediately obtain the following corollaries.

\begin{corollary}
\label{cor:fH}
If the complexification of \eqref{eqn:fH} is meromorphically integrable
 near $(x_1,x_2)=(0,0)$, then so is Eq.~\eqref{eqn:fHp0} near $r=0$.
\end{corollary}

\begin{corollary}
\label{cor:dH}
If the complexification of \eqref{eqn:dH} is meromorphically integrable
 near $(x_1,x_2)=(0,0)$ and near $(x_3,x_4)=(0,0)$,
 then so is Eq.~\eqref{eqn:dHp0} near $r_1=0$ and near $r_2=0$, respectively.
\end{corollary}

We easily see that Eqs.~\eqref{eqn:fHp} and \eqref{eqn:dHp} satisfy
 $\tilde{g}_j(0,y)\neq 0$ for almost all $y\in\tilde{D}$ for some $j=1,\ldots,m$
 and $\Theta(0,y)\neq 0$ for any $y\in\tilde{D}$.

\begin{remark}
The converses of Corollaries~$\ref{cor:fH}$ and $\ref{cor:dH}$ do not necessarily hold.
Actually, even if Eqs.~\eqref{eqn:fHp0} and \eqref{eqn:dHp0} have first integrals,
 then the first integrals may not be meromorphic %or rational 
 for the complexifications of \eqref{eqn:fH} and \eqref{eqn:dH}, respectively.
A similar statement is also true for commutative vector fields.
\end{remark}

\section{Nonintegrability of planar polynomial vector fields}

\subsection{General Results}
Consider planar polynomial vector fields of the form
\begin{equation}\label{eqn:pls}
\dot{\xi}=P(\xi,\eta),\quad
\dot{\eta}=Q(\xi,\eta),\quad
(\xi,\eta)\in\Cset^2,
\end{equation}
where $P(\xi,\eta)$ and $Q(\xi,\eta)$ are polynomials.
Let $\Gamma: \eta-\varphi(\xi)=0$ be an integral curve of \eqref{eqn:pls}
 where $\varphi(\xi)$ is assumed to be a rational function of $\xi$.
So $\Gamma$ represents a rational solution to the first-order differential equation
\begin{equation} \label{eqn:fol}
\eta^\prime=\frac{Q(\xi,\eta)}{P(\xi,\eta)}=:R(\xi,\eta),
\end{equation}
which defines a foliation associated with \eqref{eqn:pls} (or its orbits),
 where the prime denotes differentiation with respect to $\xi$
 and $R(\xi,\eta)$ is rational in $\xi$ and $\eta$.

Let $\phi(\xi,\eta)$ denote the (nonautonomous) flow
 of the one-dimensional system \eqref{eqn:fol} with $\phi(\xi_0,\eta)=\eta$ for $\xi_0$ fixed,
 and let $(\xi_0,\eta_0)$ be a point on $\Gamma$, i.e., $\eta_0=\varphi(\xi_0)$.
We are interested in the variation of $\phi(\xi,\eta)$
 with respect to $\eta$ around $\eta=\eta_0$ at $\xi=\xi_0$, which is expressed as
\[
\phi(\xi,\eta)=\varphi(\xi)+\frac {\partial \phi}{\partial \eta}(\xi,\eta_0)(\eta-\eta_0)
+\frac 12\frac {\partial^2 \phi}{\partial \eta^2}(\xi,\eta_0)(\eta-\eta_0)^2+\cdots.
\]
So we want to compute the above Taylor expansion coefficients
\[
\varphi_k(\xi)=\frac {\partial^k \phi}{\partial \eta^k}(\xi,\eta_0),\quad
k\in\Nset,
\]
which are solutions to the equations in variation. 
Let
\begin{equation}
\kappa_k(\xi):=\frac{\partial^k R}{\partial \eta^k}(\xi,\varphi(\xi)),\quad
k\in\Nset.
\label{eqn:kappa}
\end{equation}
Note that $\kappa_k(\xi)$ is rational for any $k\in\Nset$.
{\renewcommand{\theequation}{$\mathrm{VE}_1$}
The \emph{first-} and \emph{second-order variational equations}
 ($\mathrm{VE}_1$ and $\mathrm{VE}_2$) are given by
\begin{equation}\label{VE1}
\varphi_1'=\kappa_1(\xi)\varphi_1
\end{equation}
}and
{\renewcommand{\theequation}{$\mathrm{VE}_2$}
\begin{equation}\label{VE2}
 \varphi_1^\prime=\kappa_1(\xi)\varphi_1,\quad
 \varphi_2^\prime=\kappa_1(\xi)\varphi_2+ \kappa_2(\xi)\varphi_1^2,
\end{equation}
}respectively.
The $\mathrm{VE}_1$ is linear
 but the $\mathrm{VE}_2$ is nonlinear.
Letting $\chi_{21}:=\varphi_1^2$ and $\chi_{22}:=\varphi_2$,
 we can linearize the $\mathrm{VE}_2$ as
{\renewcommand{\theequation}{$\mathrm{LVE}_2$}
\begin{equation}\label{LVE2}
%\begin{array}{ll}
 \chi_{21}^\prime=2\kappa_1(\xi)\chi_{21},\quad
 \chi_{22}^\prime=\kappa_1(\xi)\chi_{22}+\kappa_2(\xi)\chi_{21},
% \end{array}
\end{equation}
}and refer to it as the \emph{second-order linearized variational equation}
 ($\mathrm{LVE}_2$).
We also refer to the $\mathrm{VE}_1$ as the $\mathrm{LVE}_1$. 
In a similar way, for any $k>2$,
 we obtain the \emph{$k$th-order variational equation} $\mathrm{VE}_k$ as
{\renewcommand{\theequation}{$\mathrm{VE}_k$}
\begin{align}\label{VEk}
&
\varphi'_1=\kappa_1(\xi)\varphi_1,\quad
\varphi'_2=\kappa_1(\xi)\varphi_2+\kappa_2(\xi)\varphi_1^2,
\quad\ldots,\notag\\
&
\varphi'_k=\kappa_1(\xi)\varphi_k%+\kappa_2(\xi)(k\varphi_1\varphi_{k-1}+\cdots)
 +\cdots+\tfrac{1}{2}k(k-1)\kappa_{k-1}(\xi)\varphi_1^{k-2}\varphi_2 +\kappa_k(\xi)\varphi_1^k.
\end{align}
}We can also linearize the $\mathrm{VE}_k$ as
{\renewcommand{\theequation}{$\mathrm{LVE}_k$}
\begin{align}\label{LVEk}
&
\chi_{k1}^\prime=k\kappa_1(\xi)\chi_{k1},\quad
\chi_{k2}^\prime=(k-1)\kappa_1(\xi)\chi_{k2}+\kappa_2(\xi)\chi_{k1},\quad\ldots,\notag\\
&
\chi_{kk}^\prime=\kappa_1(\xi)\chi_{kk}+\cdots+\tfrac{1}{2}k(k-1)\kappa_{k-1}(\xi)\chi_{k2}+\kappa_k(\xi)\chi_{k1},
\end{align}
}and refer to it as the \emph{$k$th-order linearized variational equation}
 ($\mathrm{LVE}_k$),
 where $\chi_{k1}=\varphi_1^k,\chi_{k2}=\varphi_1^{k-2}\varphi_2,\ldots,\chi_{kk}=\varphi_k$.
\setcounter{equation}{3}
We observe that the $\mathrm{LVE}_k$ has a two-dimensional subsystem
\begin{equation}\label{eqn:sLVEk}
\chi_{k1}'=k \kappa_1(\xi)\chi_{k1},\quad
\chi_{kk}'=\kappa_1(\xi)\chi_{kk} +\kappa_k(\xi)\chi_{k1}
\end{equation}
for any $k\ge 2$.

Let $G_k$ be the differential Galois group of the $\mathrm{LVE}_k$
 and let $G_k^0$ be its identity component.
Using the result of Ayoul and Zung \cite{AZ} based on \cite{M,MR,MRS},
 we have the following theorem  \cite{ALMP2}.

\begin{theorem}%[Ayoul-Zung \cite{AZ}]
\label{thm:a}
Assume that the $\mathrm{VE}_1$ has no irregular singularity at infinity
 and the planar polynomial vector field \eqref{eqn:pls}
 is meromorphically integrable in a neighbourhood of $\Gamma$.
Then for any $k\geq 1$ the identity component $G_k^0$ is abelian.
\end{theorem}

The statement of the  above theorem also holds in a more general setting.
See \cite{AZ,M,MR,MRS} for the details.
Obviously, $G_1$ and $G_1^0$ are subgroups of $\mathbb{C}^\ast$ and abelian.
However, $G_k$ and $G_k^0$ may be non-abelian for $k\ge 2$.

Let
\begin{equation}
\Omega(\xi)=\exp\left(\int\kappa_1(\xi)\,\d\xi\right),\quad
\theta_k(\xi)=\int\kappa_k(\xi)\Omega(\xi)^{k-1}\d\xi
\label{eqn:omega}
\end{equation}
for $k\ge 2$.
The subsystem \eqref{eqn:sLVEk} of the $\mathrm{LVE}_k$
 has two linearly independent solutions
 $(\chi_{k1},\chi_{kk})=(0,\Omega(\xi))$ and $(\Omega(\xi)^k,\Omega(\xi)\theta_k(\xi))$.
Let $\tilde{G}$ be the differential Galois group of \eqref{eqn:sLVEk}
 and $\tilde{G}^0$ be its identity component.
We have the following criterion for $G_k^0$ to be non-abelian.

\begin{lemma}
\label{lem:3a}
Suppose that the following conditions hold for some $k\ge 2$:
\begin{enumerate}
\item[\rm(H1)]
$\Omega(\xi)$ is transcendental;
\item[\rm(H2)]
$\theta_k(\xi)/\Omega(\xi)^{k-1}$ is not rational.
\end{enumerate}
Then the identity component $G_k^0$ is not abelian.
\end{lemma}

\begin{proof}
Assume that conditions~(H1) and (H2) hold.
Let $\sigma\in\tilde{G}$.
We compute
\[
\frac{\sigma(\Omega(\xi))'}{\sigma(\Omega(\xi))}
=\sigma\left(\frac{\Omega'(\xi)}{\Omega(\xi)}\right)
=\sigma(\kappa_1(\xi))=\kappa_1(\xi)=\frac{\Omega'(\xi)}{\Omega(\xi)},
\]
which yields
\begin{equation}
\sigma(\Omega(\xi))=C_1\Omega(\xi),\quad
 C_1\in\Cset^\ast.
\label{eqn:lem3ap}
\end{equation}
So we have
\[
\sigma(\theta_k(\xi))'
 =\sigma(\kappa_k(\xi)\Omega(\xi)^{k-1})
 =C_1^{k-1}\kappa_k(\xi)\Omega(\xi)^{k-1}
 =C_1^{k-1}\theta_k'(\xi),
\]
so that for some $C_2\in\Cset$
\[
\sigma(\theta_k(\xi))=C_1^{k-1}\theta_k(\xi)+C_2.
\]

Assume that $C_2=0$ for any $\sigma\in\tilde{G}$.
Let $w(\xi)=\theta_k'(\xi)/\theta_k(\xi)=\kappa_k(\xi)\Omega(\xi)^{k-1}/\theta_k(\xi)$.
By the hypothesis, $w(\xi)$ is not rational.
However, we have
\[
\sigma(w(\xi))=\frac{\sigma(\theta_k'(\xi))}{\sigma(\theta_k(\xi))}=w(\xi),
\]
which means that $w(\xi)\in\Cset(\xi)$.
Thus, we have a contradiction.
Hence, $C_2\neq 0$ for some $\sigma\in\tilde{G}$.
Taking $(\chi_{k1},\chi_{kk})=(0,\Omega(\xi))$ and $(\Omega(\xi)^k,\Omega(\xi)\theta_k(\xi))$
 as fundamental solutions to \eqref{eqn:sLVEk}
 and noting that $\Omega(\xi)$ is transcendental, we see that
\[
\tilde{G}\cong\left\{
\begin{pmatrix}
c_1 & c_2\\
0 & c_1^k
\end{pmatrix}
\bigg| c_1\in\Cset^\ast,c_2\in\Cset\right\}.
\]
Hence, $\tilde{G}^0=\tilde{G}$ is not commutative.
This yields the conclusion.
\end{proof}

Let $\kappa_k(\xi)=\kappa_{k\n}(\xi)/\kappa_{k\d}(\xi)$ for $k\in\Nset$,
 where $\kappa_{k\n}(\xi)$ and $\kappa_{k\d}(\xi)$ are relatively prime polynomials
 and $\kappa_{k\d}(\xi)$ is monic.
We see that if $\deg(\kappa_{k\d})>\deg(\kappa_{k\n})$,
 then $\kappa_1(1/\xi)/\xi$ is holomorphic at $\xi=0$
 so that the $\mathrm{VE}_1$ and consequently the $\mathrm{LVE}_k$
 have no irregular singularity at infinity for $k\ge 2$.
Using Theorem~\ref{thm:a} and Lemma~\ref{lem:3a},
 we immediately obtain the following theorem.

\begin{theorem}
\label{thm:c}
Suppose that $\deg(\kappa_{1\d})>\deg(\kappa_{1\n})$
 and conditions~{\rm(H1)} and {\rm(H2)} hold for $k\ge 2$.
Then the planar polynomial vector field \eqref{eqn:pls}
 is meromorphically nonintegrable in a neighbourhood of $\Gamma$.
\end{theorem}

\begin{remark}
\label{rmk:3b}
Suppose that condition~{\rm(H1)} does not hold.
Then $C_1$ in \eqref{eqn:lem3ap} can only take finitely many values, so that
\[
G_k^0\subset\left\{\left.
\begin{pmatrix}
1 & c_{12} & c_{13} & \cdots & c_{1k}\\
0 & 1 & c_{23} & \cdots & c_{2k}\\
\vdots & \ddots & \ddots & \ddots & \vdots\\
0 & \cdots & 0 & 1 & c_{k-1,k} \\
0 & \cdots & 0 & 0 & 1
\end{pmatrix}
\right| c_{12},\ldots,c_{k-1,k}\in\Cset\right\}.
\]
Thus $G_k^0$ is abelian.
\end{remark}

If the variational equations have irregular singularities at infinity,
 then an obstruction for the existence of (meromorphic) first integrals
 and commutative vector fields may appear at infinity
 when the phase space is compactified.
In such a case we can only discuss ``rational'' nonintegrability
 instead of meromorphic one \cite{M,MR}.
Moreover, if $\deg(\kappa_{1\d})\le\deg(\kappa_{1\n})$,
 then the $\mathrm{VE}_1$ and consequently the $\mathrm{LVE}_k$
 have an irregular singularity at infinity for $k\ge 2$.
Rational nonintegrability of \eqref{eqn:pls}
 in this situation was extensively discussed in \cite{ALMP2}.

\subsection{Criteria for condition~(H2)}

It is often difficult to check condition~{\rm(H2)} directly
 in application of Theorem~\ref{thm:c}
 although it does not hold in only special cases.
So we give useful criteria for condition~{\rm(H2)} below.
They are extensively used in our proofs of the main theorems
 in Sections~4 and 5. 
We begin with the following lemma.

\begin{lemma}
\label{lem:3b}
If condition~{\rm(H2)} does not hold,
 i.e., $\theta_k(\xi)/\Omega(\xi)^{k-1}\in\Cset(\xi)$, for $k\ge 2$,
 then there exist $C_3\,(\neq 0)\in\Cset$,
 $n\in\Zset_{\ge 0}:=\Nset\cup\{0\}$, $a_j\in\Zset\setminus\{0\}$
 and $\xi_j\in\Cset$, $j=1,\ldots,n$, with $\xi_j\neq\xi_\ell$ for $j\neq\ell$,
 such that
\begin{equation}
\kappa_k(\xi)
=\frac{C_3\hat{\kappa}_k(\xi)}
 {\displaystyle\kappa_{1\d}(\xi)\prod_{j=1}^n(\xi-\xi_j)^{a_j+1}},
\label{eqn:lem3b}
\end{equation}
where
\begin{equation}
 \hat{\kappa}_k(\xi)=(k-1)\kappa_{1\n}(\xi)\prod_{j=1}^n(\xi-\xi_j)
 -\kappa_{1\d}(\xi)\sum_{j=1}^n a_j\prod_{\ell\neq j}(\xi-\xi_\ell).
\label{eqn:lem3b0}
\end{equation}
In particular, if $n=0$,
 then Eq.~\eqref{eqn:lem3b} reduces to $\kappa_k(\xi)=C_3\kappa_1(\xi)$.
\end{lemma}

\begin{proof}
Let $w(\xi)=\theta_k'(\xi)/\theta_k(\xi)$
 as in the proof of Lemma~\ref{lem:3a}.
We easily have
\[
\theta_k(\xi)=C_3\exp\left(\int w(\xi)\d\xi\right)
\]
for some constant $C_3\neq 0$.
Hence,
\begin{equation}
\theta_k'(\xi)=C_3 w(\xi)\exp\left(\int w(\xi)\d\xi\right).
\label{eqn:lem3b1}
\end{equation}
On the other hand, by \eqref{eqn:omega}
\begin{equation}
\theta_k'(\xi)=\kappa_k(\xi)\Omega(\xi)^{k-1}
 =\kappa_k(\xi)\exp\left((k-1)\int\kappa_1(\xi)\d\xi\right).
\label{eqn:lem3b2}
\end{equation}

Assume that condition~(H2) does not hold.
Then $w(\xi)=\kappa_k(\xi)\Omega(\xi)^{k-1}/\theta_k(\xi)$ is rational.
Comparing \eqref{eqn:lem3b1} and \eqref{eqn:lem3b2},
 we cannot conclude that $w(\xi)=(k-1)\kappa_1(\xi)$ but obtain
\[
w(\xi)
 =(k-1)\kappa_1(\xi)-\sum_{j=1}^n\frac{a_j}{\xi-\xi_j}
 =C_3^{-1}\kappa_k(\xi)\prod_{j=1}^n(\xi-\xi_j)^{a_j}
\]
where $n\in\Zset_{\ge 0}$, $a_j\in\Zset\setminus\{0\}$
 and $\xi_j\in\Cset$, $j=1,\ldots,n$, with $\xi_j\neq\xi_\ell$ for $j\neq\ell$,
 since $\kappa_1(\xi),\kappa_k(\xi),w(\xi)\in\Cset(\xi)$.
This yields the desired result.
\end{proof}

This lemma means that $\kappa_k(\xi)$ has the very special form \eqref{eqn:lem3b}
 with \eqref{eqn:lem3b0} if condition~(H2) does not hold,
 and it is useful to determine whether condition~(H2) holds.
It is clear that the polynomial $\hat{\kappa}_k(\xi)$ has a zero at $\xi=\xi_j$
 if $\kappa_{1\d}(\xi_j)=0$, and $\hat{\kappa}_k(\xi_j)\neq 0$ otherwise.
For $k\ge 2$, we write
\begin{equation}
\kappa_{k\d}(\xi)
 =\kappa_{1\d}(\xi)\prod_{j=1}^{n_1}(\xi-\xi_{1j})^{a_{1j}}
 \prod_{j=1}^{n_k}(\xi-\xi_{kj})^{a_{kj}},
\label{eqn:kd}
\end{equation}
where $n_\ell\in\Zset_{\ge 0}$,
 $\xi_{\ell j}\in\Cset$ and $a_{\ell j}\in\Zset\setminus\{0\}$,
 $j=1,\ldots,n_\ell$, if $n_\ell>0$ for $\ell=1,k$,
 such that $\xi_{1j}$ is a root of $\kappa_{1\d}(\xi)$ but $\xi_{kj}$ is not,
 and $\xi_{\ell j_1}\neq\xi_{\ell j_2}$ if $j_1\neq j_2$.
Note that $a_{kj}>0$, $j=1,\ldots,n_k$,
 and $a_{1j}\ge -b_{1j}$ but $a_{1j}\neq 0$, $j=1,\ldots,n_1$,
 if $n_k$ and $n_1$, respectively, are positive,
 where $b_{1j}$ is the multiplicity of the zero $\xi_{1j}$ for $\kappa_{1\d}(\xi)$
 since $\kappa_{k\d}(\xi)$ is a polynomial and $\kappa_{1\d}(\xi_{kj})\neq 0$, $j=1,\ldots,n_k$.
When $n_1>0$, let
\begin{equation}
\bar{\kappa}_{kb}(\xi)=(k-1)\kappa_{1\n}(\xi)\prod_{j=1}^{n_1}(\xi-\xi_{1j})
 -\kappa_{1\d}(\xi)\sum_{j=1}^{n_1}(a_{1j}+b_j-1)\prod_{\ell\neq j}(\xi-\xi_{1\ell}),
\label{eqn:bkappa}
\end{equation}
where $b=(b_1,\ldots,b_{n_1})$ with $b_j\in\Nset$, $j=1,\ldots.n_1$.
Obviously, $\bar{\kappa}_{kb}(\xi)$ has a zero at $\xi=\xi_{1j}$ like $\hat{\kappa}_k(\xi)$
 by $\kappa_{1\d}(\xi_{1j})=0$.
We see that if $b_{1j}>1$, i.e., the zero $\xi_{1j}$ is not simple for $\kappa_{1\d}(\xi)$,
 then it is simple for $\bar{\kappa}_{kb}(\xi)$ with any $b\in\Nset^{n_1}$,
 since  $\kappa_{1\n}(\xi_{1j})\neq 0$.

\begin{lemma}
\label{lem:3c}
Suppose that condition~{\rm(H2)} does not hold and $n_1>0$,
 and fix $j\in\{1,\ldots,n_1\}$.
\begin{enumerate}
\item[\rm(i)]
If the zero $\xi_{1j}$ is not simple for $\bar{\kappa}_{kb_0}(\xi)$ with some $b_0\in\Nset^{n_1}$,
 then it is simple for $\kappa_{1\d}(\xi)$ and $\bar{\kappa}_{kb}(\xi)$ with $b_j\neq b_{j0}$,
 where $b_{j0}$ is the $j$-th element of $b_0$ for $j=1,\ldots,n_1$.
\item[\rm(ii)]
If the zero $\xi_{1j}$ is simple for $\bar{\kappa}_{kb_0}(\xi)$ with some $b_0\in\Nset^{n_1}$,
 then so is it for $\bar{\kappa}_{kb}(\xi)$ with $b_j=b_{j0}$.
\end{enumerate}
\end{lemma}

\begin{proof}
Assume that the zero $\xi_{1j}$ of $\bar{\kappa}_{kb_0}(\xi)$ is not simple.
Then the zero $\xi_{1j}$ is simple  for $\kappa_{1\d}(\xi)$, i.e., $b_{1j}=1$,
 or else it is simple for $\bar{\kappa}_{kb_0}(\xi)$ as stated above.
Hence, if $b_j\neq b_{j0}$, then the zero $\xi_{1j}$ is simple for
\begin{equation}
\bar{\kappa}_{kb}(\xi)=\bar{\kappa}_{kb_0}(\xi)
 -\kappa_{1\d}(\xi)\sum_{m=1}^{n_1}(b_m-b_{m0})\prod_{\ell\neq m}(\xi-\xi_{1\ell}),
\label{eqn:lem3c}
\end{equation}
since it is a simple zero of the second term.
Thus, we obtain part~(i).

We next assume that the zero $\xi_{1j}$ is simple for $\bar{\kappa}_{kb_0}(\xi)$.
As easily seen, if $b_j=b_{j0}$,
then $\xi=\xi_{1j}$ is  at least a double zero of the second term in \eqref{eqn:lem3c}.
This means part~(ii).
\end{proof}

Define the polynomial
\begin{align}
\rho_k(\xi)
=& %\frac{1}{\displaystyle\prod_{j=1}^{n_1}(\xi-\xi_{1j})^{b_j}}
 (k-1)\kappa_{1\n}(\xi)\prod_{j=1}^{n_k}(\xi-\xi_{kj})
 -\kappa_{1\d}(\xi)\sum_{j=1}^{n_k}(a_{kj}-1)\prod_{\ell\neq j}(\xi-\xi_{k\ell})\notag\\
& -\frac{\displaystyle\kappa_{1\d}(\xi)\prod_{j=1}^{n_k}(\xi-\xi_{1k})}
 {\displaystyle\prod_{j=1}^{n_1}(\xi-\xi_{1j})}
  \sum_{j=1}^{n_1}a_{1j}\prod_{\ell\neq j}(\xi-\xi_{1\ell}).
\label{eqn:rhok}
\end{align}
%Note that $\rho_k(\xi_{1j_1}),\rho_k(\xi_{kj_k})\neq 0$
% for $j_1=1,\ldots,n_1$ and $j_k=k,\ldots,n_k$.
Let $\bar{\rho}_k(\xi)$ and $\tilde{\rho}_k(\xi)$ be the quotient and remainder,
 respectively, when $\kappa_{k\n}(\xi)$ is divided by $\rho_k(\xi)$.
So $\kappa_{k\n}(\xi)=\bar{\rho}_k(\xi)\rho_k(\xi)+\tilde{\rho}_k(\xi)$.
Let $\bar{n}\in\Zset_{\ge 0}$ be the number of distinct roots of $\bar{\rho}_k(\xi)$,
 and let $\bar{\xi}_j\in\Cset$ and $\bar{a}_j\in\Nset$, $j=1,\ldots,\bar{n}$,
 denote its roots and multiplicities, respectively, if $\bar{n}\ge 1$:
\begin{equation}
\bar{\rho}_k(\xi)=\bar{C}\prod_{j=1}^{\bar{n}}(\xi-\bar{\xi}_j)^{\bar{a}_j},
\label{eqn:brhok}
\end{equation}
where $\bar{C}\in\Cset$ is a nonzero constant.
If $\deg(\kappa_{k\n})\le\deg(\rho_k)$,
 then we set $\bar{n}=0$ and $\bar{\rho}_k(\xi)\equiv\bar{C}_0$,
 where $\bar{C}_0\in\Cset$ is a constant which may be zero.
We also consider the first-order differential equation
\begin{equation}
\left(\kappa_{1\d}(\xi)\prod_{j=1}^{n_k}(\xi-\xi_{kj})\right)z'+\rho_k(\xi)z
 =\kappa_{k\n}(\xi).
\label{eqn:prop3b}
\end{equation}
Let $\rho_{k0}$ be the leading coefficient of $\rho_k(\xi)$
 and let $\bar{\kappa}_k(\xi)=\bar{\kappa}_{kb}(\xi)$
 with $b=(1,\ldots,1)\in\Nset^{n_1}$, i.e.,
\[
\bar{\kappa}_k(\xi)=(k-1)\kappa_{1\n}(\xi)\prod_{j=1}^{n_1}(\xi-\xi_{1j})
 -\kappa_{1\d}(\xi)\sum_{j=1}^{n_1}a_{1j}\prod_{\ell\neq j}(\xi-\xi_{1\ell}).
\]
Using Lemmas~\ref{lem:3b} and \ref{lem:3c},
 we obtain some effective criteria for condition~(H2) as follows.

\begin{proposition}
\label{prop:3b}
Let $k\ge 2$.
Suppose that $\kappa_{1\n}(\xi),\kappa_{k\n}(\xi)\not\equiv 0$
 and $\kappa_{k\n}(\xi_{1j})\neq 0$, $j=1,\ldots,n_1$.
If one of the following conditions holds, then condition {\rm(H2)} holds.
\begin{enumerate}
\item[\rm(i)]
$a_{kj}=1$ for some $j=1,\ldots,n_k;$
\item[\rm(ii)]
For each $j=1,\ldots,n_1$,
 the zero $\xi_{1j}$ is not simple for $\bar{\kappa}_k(\xi)$
 or simple for $\bar{\kappa}_{kb}(\xi)$ with some $b\in\Nset^{n_1}$ for each $b_j\in\Nset$.
\end{enumerate}
Moreover, if $n_1>0$, then assume that for $j=1,\ldots,n_1$
 the zero $\xi_{1j}$ of $\bar{\kappa}_{kb}(\xi)$ is simple
 when $b_j>1$.
If one of the following conditions holds, then condition {\rm(H2)} holds$:$
\begin{enumerate}
\item[\rm(iii)]
Eq.~\eqref{eqn:prop3b} does not have a polynomial solution
 that has no root at $\xi=\xi_{1j}$ and $\xi_{k\ell}$
 for any $j=1,\ldots,n_1$ and $\ell=1,\ldots,n_k;$
\item[\rm(iv)]
$\bar{n}=0$,
\begin{enumerate}
\item[\rm(iva)]
$\bar{\rho}_k(\xi)\equiv 0$ or $\tilde{\rho}_k(\xi)\not\equiv 0;$
\item[and \rm(ivb)]
$\deg(\kappa_{1\d})+n_k\neq\deg(\rho_k)+1$ or $-\rho_{k0}\notin\Nset;$
\end{enumerate}
\item[\rm(v)]
$\bar{n}>0$ and $\deg(\kappa_{1\d})+n_k>\max(\deg(\kappa_{k\n}),\deg(\rho_k)+1);$
\item[\rm(vi)]
$\bar{n}>0$, $\deg(\kappa_{1\d})+n_k<\deg(\rho_k)-\deg(\bar{\rho}_k)+1$ and
\begin{enumerate}
\item[\rm(via)]
$\bar{\rho}_k(\xi)$ has a root at $\xi=\xi_{1j}$ or $\xi_{k\ell}$
 for some $j=1,\ldots,n_1$ or $\ell=1,\ldots,n_k;$
\item[or \rm(vib)]
$\displaystyle
\tilde{\rho}_k(\xi)\not\equiv\kappa_{1\d}(\xi)\bar{\rho}_k'(\xi)\prod_{j=1}^{n_k}(\xi-\xi_{kj})$.
\end{enumerate}
\end{enumerate}
\end{proposition}

Recall that $n_1,n_k,\xi_{1j},\xi_{kj}$ were defined in \eqref{eqn:kd}
 and $\bar{n}$ is the number of distinct roots of $\bar{\rho}_k(\xi)$.

\begin{proof}
Assume that $\kappa_{1\n}(\xi),\kappa_{k\n}(\xi)\not\equiv 0$,
 $\kappa_{k\n}(\xi_{1j})\neq 0$ and condition~(H2) does not hold.
Then by Lemma~\ref{lem:3b} Eq.~\eqref{eqn:lem3b} holds
 for $n\in\Zset_{\ge 0}$, $a_j\in\Zset\setminus\{0\}$
 and $\xi_j\in\Cset$, $j=1,\ldots,n$, with $\xi_j\neq\xi_\ell$ for $j\neq\ell$.
Comparing \eqref{eqn:lem3b} and \eqref{eqn:kd}
 and noting that $\kappa_{k\n}(\xi_{1j})\neq 0$ for $j=1,\ldots,n_1$,
 we can take
\begin{alignat}{3}
&
\xi_j=\xi_{1j}
&\quad&
\mbox{for $1\le j\le n_1$},\notag\\
&
\xi_{j+n_1}=\xi_{kj},\quad
a_{j+n_1}=a_{kj}-1\ge 1
&&
\mbox{for $1\le j\le n_k$},\label{eqn:prop3b1a}\\
&
\xi_{j+n_1+n_k}=\hat{\xi}_j,\quad
a_{j+n_1+n_k}+1\le 0
&&
\mbox{for $1\le j<\hat{n}=n-n_1-n_k$},\notag
\end{alignat}
where $\hat{\xi}_j\in\Cset$, $j=1,\ldots,\hat{n}$,
 such that $\hat{\xi}_j\neq\xi_{1\ell_1},\xi_{k\ell_k}$
 for any $\ell_1=1,\ldots,n_1$ and $\ell_k=1,\ldots,n_k$.
Here we have used the fact that
 $\hat{\kappa}_k(\xi),\kappa_{1\d}(\xi)\neq 0$
 at $\xi=\xi_{kj}$ and $\hat{\xi}_\ell$
 for $j=1,\ldots,n_k$ and $\ell=1,\ldots,\hat{n}$.
 If $n_1=0$ or $n_k=0$,
 then the corresponding relation in \eqref{eqn:prop3b1a} is ignored.
In particular, $a_{kj}\ge 2$, $j=1,\ldots,n_k$.
So condition~(i) does not occur.

Assume that $n_1>0$.
Since $\kappa_{k\n}(\xi_{1j})\neq 0$,
 one has $a_{1j}\ge 1$ for $j=1,\ldots,n_1$.
Let $\hat{a}_j=-a_{j+n_1+n_k}\ge1$ for $j=1,\ldots,\hat{n}$,
 and let $\hat{b}_j\in\Nset$ be the multiplicity of the zero $\xi_{1j}$
 of $\hat{\kappa}_k(\xi)$ for $j=1,\ldots,n_1$.
Again, via \eqref{eqn:lem3b} and \eqref{eqn:kd},
\[
a_j=a_{1j}+\hat{b}_j-1\ge 1,
\quad
1\le j\le n_1,
\]
so that Eq.~\eqref{eqn:lem3b0} becomes
\begin{align}
\hat{\kappa}_k(\xi)
=&\prod_{j=1}^{n_k}(\xi-\xi_{kj})\prod_{j=1}^{\hat{n}}(\xi-\hat{\xi}_{j})
\biggl((k-1)\kappa_{1\n}(\xi)\prod_{j=1}^{n_1}(\xi-\xi_{1j})\notag\\
&\hspace{10em}
-\kappa_{1\d}(\xi)\sum_{j=1}^{n_1}(a_{1j}+\hat{b}_j-1)\prod_{\ell\neq j}(\xi-\xi_{1\ell})
 \biggr)\notag\\
&
-\kappa_{1\d}(\xi)\prod_{j=1}^{n_1}(\xi-\xi_{1j})
 \biggl(\prod_{j=1}^{\hat{n}}(\xi-\hat{\xi}_j)
 \sum_{j=1}^{n_k}(a_{kj}-1)\prod_{\ell\neq j}(\xi-\xi_{k\ell})\notag\\
&\hspace{10em}
 -\prod_{j=1}^{n_k}(\xi-\xi_{kj})
 \sum_{j=1}^{\hat{n}}\hat{a}_{j}\prod_{\ell\neq j}(\xi-\xi_{k\ell})\biggr).
\label{eqn:prop3b1}
\end{align}
Suppose that condition~(ii) holds.
Then it follows from Lemma~\ref{lem:3c} that
 the zeros $\xi_{1j}$ of $\bar{\kappa}_{kb}(\xi)$, $j=1,\ldots,n_1$,
 are all simple for $b\neq(1,\ldots,1)\in\Nset^{n_1}$
 although it may be simple even for $b=(1,\ldots,1)$.
Hence, if $\hat{b}_j>1$ for some $j\in\{1,\ldots,n_1\}$,
 then we see via \eqref{eqn:prop3b1} that
 $\hat{\kappa}_k(\xi)$ has a simple zero at $\xi=\xi_{1j}$,
 since so does
\[
\bar{\kappa}_{k\hat{b}}(\xi)
=(k-1)\kappa_{1\n}(\xi)\prod_{j=1}^{n_1}(\xi-\xi_{1j})\notag\\
 -\kappa_{1\d}(\xi)\sum_{j=1}^{n_1}(a_{1j}+\hat{b}_j-1)\prod_{\ell\neq j}(\xi-\xi_{1\ell})
\]
with $\hat{b}=(\hat{b}_1,\ldots,\hat{b}_{n_1})$
 as well as $\kappa_{1\d}(\xi_{1j})=0$.
This yields a contradiction, so that $\hat{b}_j=1$, i.e.,
 $\hat{\kappa}_k(\xi)$ has a simple zero at $\xi=\xi_{1j}$ for $j=1,\ldots,n_1$.
Let $\hat{b}_j=1$, $j=1,\ldots,n_1$, in \eqref{eqn:prop3b1}.
Then the zeros $\xi=\xi_{1j}$, $j=1,\ldots,n_1$,
 are all simple for $\bar{\kappa}_k(\xi)$
 since they are not simple for $\hat{\kappa}_k(\xi)$ if not.
Thus, condition~(ii) does not occur.

We now assume that $n_1>0$
 and $\bar{\kappa}_{kb}(\xi)$ has a simple zero at $\xi=\xi_{1j}$ if $b_j>1$.
From the above argument
 we see that for $j=1,\ldots,n_1$ the zero $\xi=\xi_{1j}$ of $\hat{\kappa}_k(\xi)$
 is simple, i.e., $\hat{b}_j=1$, so that $a_j=a_{1j}\neq 0$.
Using \eqref{eqn:lem3b} and \eqref{eqn:prop3b1a}, we have
\[
\kappa_{k\n}(\xi)
=\frac{\displaystyle C_3\hat{\kappa}_k(\xi)
 \prod_{j=1}^{\hat{n}}(\xi-\hat{\xi}_{j})^{\hat{a}_j-1}}
 {\displaystyle\prod_{j=1}^{n_1}(\xi-\xi_{1j})}.
\]
Substituting \eqref{eqn:lem3b0} into the above equation and using \eqref{eqn:rhok},
 we obtain
\begin{align}
\kappa_{k\n}(\xi)
=&
 C_3\Biggl((k-1)\kappa_{1\n}(\xi)\prod_{j=1}^{n_k}(\xi-\xi_{kj})
 \prod_{j=1}^{\hat{n}}(\xi-\hat{\xi}_{j})^{\hat{a}_j}\notag\\
&
-\frac{\kappa_{1\d}(\xi)}{\displaystyle\prod_{j=1}^{n_1}(\xi-\xi_{1j})}
 \prod_{j=1}^{n_k}(\xi-\xi_{kj})
 \prod_{j=1}^{\hat{n}}(\xi-\hat{\xi}_j)^{\hat{a}_j}
 \sum_{j=1}^{n_1}a_{1j}\prod_{\ell\neq j}(\xi-\xi_{1\ell})\notag\\
&
-\kappa_{1\d}(\xi)\biggl(
 \prod_{j=1}^{\hat{n}}(\xi-\hat{\xi}_j)^{\hat{a}_j}
 \sum_{j=1}^{n_k}(a_{kj}-1)\prod_{\ell\neq j}(\xi-\xi_{k\ell})\notag\\
& \hspace*{5em}
 -\prod_{j=1}^{n_k}(\xi-\xi_{kj})
 \prod_{j=1}^{\hat{n}}(\xi-\hat{\xi}_\ell)^{\hat{a}_j-1}
 \sum_{j=1}^{\hat{n}}\hat{a}_j\prod_{\ell\neq j}(\xi-\hat{\xi}_\ell)\biggr)\Biggr)\notag\\
=& \rho_k(\xi)\hat{\rho}_k(\xi)
 +\hat{\rho}_k'(\xi)\kappa_{1\d}(\xi)\prod_{j=1}^{n_k}(\xi-\xi_{kj})
\label{eqn:prop3b2}
\end{align}
where
\begin{equation}
\hat{\rho}_k(\xi)=C_3\prod_{j=1}^{\hat{n}}(\xi-\hat{\xi}_j)^{\hat{a}_j}.
\label{eqn:polsol}
\end{equation}
Recall that $\hat{a}_j=-a_{j+n_1+n_k}$, $j=1,\ldots,\hat{n}$.
We easily see that Eq.~\eqref{eqn:prop3b2} holds even if $n_1=0$.
Thus, $\hat{\rho}_k(\xi)$ is a polynomial solution to \eqref{eqn:prop3b},
 so that condition~(iii) does not occur.
 
It remains to show that conditions~(iv)-(vi) do not occur
 when condition~(H2) does not hold under our other assumptions.
The expression \eqref{eqn:prop3b2} gives a key for our proofs of the remaining parts.
Recall that $\bar{n}$ and $\hat{n}$ are,  respectively, the numbers of distinct roots
 of $\bar{\rho}_k(\xi)$ and $\hat{\rho}_k(\xi)$.
We need the following lemma.

\begin{lemma}
\label{lem:3d}
\
%{\setlength{\leftmargini}{21pt}
\begin{enumerate}
\item[\rm(i)]
If $\bar{n}>0$, then $\hat{n}>0$.
\item[\rm(ii)]
If $\bar{n}=0$ and one of the following conditions holds, then $\hat{n}=0\!:$
\begin{enumerate}
\item[\rm(iia)]
$\deg(\kappa_{1\d})+n_k\neq\deg(\rho_k)+1;$
\item[\rm(iib)]
$-\rho_{k0}\notin\Nset$.
\end{enumerate}
\end{enumerate}%}
\end{lemma}

\begin{proof}
Suppose that $\bar{n}>0$.
Then $\deg(\kappa_{k\n})>\deg(\rho_k)$.
However, if $\hat{n}=0$,
 then the degree of the right hand side in \eqref{eqn:prop3b2} is $\deg(\rho_k)$.
This is a contradiction.
Thus, we obtain part (i).

Suppose that $\bar{n}=0$.
Then $\deg(\kappa_{k\n})\le\deg(\rho_k)$.
If $\hat{n}>0$ and $\deg(\kappa_{1\d})+n_k\neq\deg(\rho_k)+1$,
 then the degree of the right hand side in \eqref{eqn:prop3b2} becomes
\[
\deg(\kappa_{1\d})+n_k+\deg(\hat{\rho}_k)-1
\quad\mbox{or}\quad
\deg(\rho_k)+\deg(\hat{\rho}_k),
\]
depending on whether $\deg(\kappa_{1\d})+n_k>\deg(\rho_k)+1$ or not,
 so that $\deg(\kappa_{k\n})>\deg(\rho_k)$ for both cases.
On the other hand, if $\hat{n}>0$,
 $-\rho_{k0}\notin\Nset$ and $\deg(\kappa_{1\d})+n_k=\deg(\rho_k)+1$,
 then the leading coefficient of the right hand side in \eqref{eqn:prop3b2} is
\[
C_3(\rho_{k0}+\deg(\hat{\rho}_k))\neq 0,
\]
so that its degree becomes $\deg(\rho_k)+\deg(\hat{\rho}_k)>\deg(\rho_k)$.
Thus, we have a contradiction if condition (iia) or (iib) holds.
So we obtain part (ii).
\end{proof}

We return to the proof of Proposition~\ref{prop:3b}.
Suppose that $\bar{n}>0$ and
\[
\deg(\kappa_{1\d})+n_k<\deg(\rho_k)-\deg(\bar{\rho}_k)+1.
\]
Then we have
\[
\deg(\rho_k)>\deg(\kappa_{1\d})+n_k-1,
\]
so that $\deg(\kappa_{k\n})=\deg(\rho_k)+\deg(\hat{\rho}_k)$ by \eqref{eqn:prop3b2}.
Moreover, by Lemma~\ref{lem:3d}(i), $\hat{n}>0$ and consequently $\deg(\hat{\rho}_k)>0$.
Since $\deg(\kappa_{k\n})=\deg(\rho_k)+\deg(\bar{\rho}_k)$ by definition,
 we have $\deg(\bar{\rho}_k)=\deg(\hat{\rho}_k)$, so that
\[
\deg(\kappa_{1\d})+\deg(\hat{\rho}_k)+n_k-1<\deg(\rho_k).
\]
Hence, it follows from \eqref{eqn:prop3b2} that
 when $\kappa_{k\n}(\xi)$ is divided by $\rho_k(\xi)$,
 the quotient $\bar{\rho}_k(\xi)\not\equiv 0$ is equivalent to $\hat{\rho}_k(\xi)$
 and given by \eqref{eqn:brhok}
 with $\bar{C}=C_3$, $\bar{\xi}_j=\hat{\xi}_j$, $\bar{a}_j=\hat{a}_j$ and $\bar{n}=\hat{n}$, 
 and the remainder becomes
\[
\tilde{\rho}_k(\xi)=\kappa_{1\d}(\xi)\bar{\rho}_k'(\xi)\prod_{j=1}^{n_k}(\xi-\xi_{kj}).
\]
Thus, condition~(vi) does not occur.

If $\bar{n}=0$ and condition~(ivb) holds,
 then by Lemma~\ref{lem:3d}(ii) $\hat{n}=0$,
 so that by \eqref{eqn:prop3b2}
 $\bar{\rho}_k(\xi)\equiv C_3\neq 0$ and $\tilde{\rho}_k(\xi)\equiv 0$,
 i.e., condition~(iva) does not hold.
Hence, condition~(iv) does not occur.
If $\bar{n}>0$ and $\deg(\kappa_{1\d})+n_k>\deg(\rho_k)+1$,
 then by \eqref{eqn:prop3b2}
\[
\deg(\kappa_{k\n})\ge\deg(\kappa_{1\d})+n_k
\]
since $\hat{n}>0$ by Lemma~\ref{lem:3d}(i).
Hence, condition~(v) does not occur.
We complete the proof.
\end{proof}

\begin{remark}
\label{rmk:3c}
Suppose that $\kappa_{1\n}(\xi),\kappa_{k\n}(\xi)\not\equiv 0$;
 $\kappa_{k\n}(\xi_{1j})\neq 0$, $a_{kj}\ge 2$
 and the zero $\xi_{1j}$ of $\bar{\kappa}_{kb}(\xi)$ is simple when $b_j>1$ for $j=1,\ldots,n_1$;
 and Eq.~\eqref{eqn:prop3b} has a polynomial solution of the form \eqref{eqn:polsol}
 such that $\hat{\xi}_j\neq\xi_{1\ell_1},\xi_{k\ell_k}$
 for any $j=1,\ldots,\bar{n}$, $\ell_1=1,\ldots,n_1$ and $\ell_k=1,\ldots,n_k$.
Then from the above proof we see that
\[
\frac{\theta_k'(\xi)}{\theta_k(\xi)}
 =(k-1)\kappa_1(\xi)-\sum_{j=1}^{n_1}\frac{a_{1j}}{\xi-\xi_{1j}}
 -\sum_{j=1}^{n_k}\frac{a_{kj}-1}{\xi-\xi_{kj}}
 +\sum_{j=1}^{\hat{n}}\frac{\hat{a}_j}{\xi-\hat{\xi}_j}\in\Cset(\xi).
\]
Obviously, condition~{\rm(H2)} does not hold.
\end{remark}

\section{Proof of Theorem~\ref{thm:fH1}}

We begin with Theorem~$\ref{thm:fH1}$
 for the unfolding \eqref{eqn:fH} of fold-Hopf bifurcations.

\begin{proof}[Proof of Theorem~$\ref{thm:fH1}$]
Based on Corollary~\ref{cor:fH},
 we prove the meromorphic nonintegrability of \eqref{eqn:fHp0} near the $x_3$-plane.
We set $\xi=x_3$ and $\eta=r$ and apply Theorem~\ref{thm:c} to \eqref{eqn:fHp0}
 with assistance of Proposition~\ref{prop:3b}.
Hence, we now only have to check
 $\deg(\kappa_{1\d})>\deg(\kappa_{1\n})$,
 condition~(H1) and the hypotheses of Proposition~\ref{prop:3b}.

Eq.~\eqref{eqn:fol} becomes
\begin{equation}\label{eqn:folFH}
r'=\frac{r(\alpha x_3+\nu)}{x_3^2+sr^2+\mu},
\end{equation}
where the prime represents differentiation with respect to $x_3$.
We take $r=0$ as the integral curve, i.e., $\varphi(x_3)=0$,
 and compute \eqref{eqn:kappa} as
\begin{equation}
\kappa_{2j-1}(x_3)=(2j-1)!\frac{(-s)^{j-1}(\alpha x_3+\nu)}{(x_3^2+\mu)^j},\quad
\kappa_{2j}(x_3)=0,\quad
j\in\Nset.
\label{eqn:kappafH}
\end{equation}
Recall that $s=\pm 1$.

We first consider the case of $\mu\neq 0$.
In addition, assume  that $\alpha\not\in\Qset$ or $\nu/\sqrt{-\mu}\not\in\Qset$.
Replacing $r$, $x_3$ and $\nu$
 with $\sqrt{-\mu}\,r$, $\sqrt{-\mu}\,x_3$ and $\sqrt{-\mu}\,\nu$, respectively,
 we take $\mu=-1$ and have $\alpha$ or $\nu\not\in\Qset$.
From \eqref{eqn:kappafH} we easily see
 that $\deg(\kappa_{1\d})>\deg(\kappa_{1\n})$ and compute
\[
\Omega(x_3)=\exp\left(\int\frac{\alpha x_3+\nu}{x_3^2-1}\d x_3\right)
 =(x_3+1)^{(\alpha-\nu)/2}(x_3-1)^{(\alpha+\nu)/2},
\]
so that condition~(H1) holds
 since $\alpha-\nu\not\in\Qset$ or $\alpha+\nu\not\in\Qset$.
We now only have to check the hypotheses of Proposition~\ref{prop:3b}.

Let $k=2j-1$ for $j\ge 2$.
Assume that $\alpha\neq\pm\nu$.
Then by \eqref{eqn:kappafH}
\begin{align*}
&
\kappa_{1\n}(x_3)=\alpha x_3+\nu\not\equiv 0,\quad
\kappa_{2j-1,\n}(x_3)=(2j-1)!(-s)^{j-1}(\alpha x_3+\nu)\not\equiv 0,\\
&
\kappa_{1\d}(x_3)=(x_3+1)(x_3-1),\quad
\kappa_{2j-1,\d}(x_3)=\kappa_{1\d}(x_3)(x_3+1)^{j-1}(x_3-1)^{j-1},
\end{align*}
from which $n_1=2$, $\xi_{11}=-1$, $\xi_{12}=1$, $a_{11}=a_{12}=j-1$ and $n_{2j-1}=0$.
We compute \eqref{eqn:bkappa} as
\[
\bar{\kappa}_{2j-1,b}(x_3)=(x_3+1)(x_3-1)((2(j-1)(\alpha-1)-b_1-b_2+2)x_3+2(j-1)\nu+b_1-b_2),
\]
where $b=(b_1,b_2)\in\Nset^2$.
If and only if
\begin{equation}
\alpha-\nu\neq\frac{b_1-1}{j-1}+1\quad\left(\mbox{resp.}\quad
\alpha+\nu\neq\frac{b_2-1}{j-1}+1\right),
\label{eqn:thmfH1a}
\end{equation}
then the zero $x_3=-1$ (resp. $x_3=1$) is simple for $\bar{\kappa}_{2j-1,b}(x_3)$.
Hence, if $\alpha=\nu+1$ and $\alpha=-\nu+1$, respectively,
 then the zeros $x_3=-1$ and $x_3=1$ of $\bar{\kappa}_{2j-1}(x_3)$ are double
 as well as the zeros $x_3=1$ and $x_3=-1$ of $\bar{\kappa}_{2j-1,b}(x_3)$
 are simple for any $b\in\Nset^2$.
Thus, condition~(ii) of Proposition~\ref{prop:3b} holds.

Additionally, suppose that $\alpha\neq\pm\nu+1$.
Then for some $j>1$ both conditions in \eqref{eqn:thmfH1a} hold,
 so that the zeros $x_3=\pm 1$ of $\bar{\kappa}_{2j-1,b}(\xi)$ are simple
 for any $b\in\Nset^2$ even if $\alpha$ or $\nu\in\Qset$.
Eq.~\eqref{eqn:rhok} becomes
\[
\rho_{2j-1}(x_3)=2(j-1)((\alpha-1)x_3+\nu).
\]
We see that
 $\bar{n}=0$, $\deg(\kappa_{1\d})+n_{2j-1}=\deg(\rho_{2j-1})+1=2$ and
\[
\bar{\rho}_{2j-1}(x_3)=\frac{(2j-1)!(-s)^{j-1}\alpha}{2(j-1)(\alpha-1)},\quad
\tilde{\rho}_{2j-1}(x_3)=-\frac{(2j-1)!(-s)^{j-1}\nu}{\alpha-1}
\]
if $\alpha\neq 1$, and that $\bar{n}=1$ and
\[
\deg(\kappa_{1\d})+n_{2j-1}=2>\max(\deg(\kappa_{2j-1,\n}),\deg(\rho_{2j-1})+1)=1
\]
if $\alpha=1$.
So condition~(iv) or (v) of Proposition~\ref{prop:3b} holds,
 depending on whether $\alpha\neq 1$ or not,
 where the condition
\begin{equation}
-\rho_{2j-1,0}=-2(j-1)(\alpha-1)\not\in\Nset,
\label{eqn:thmfH1b}
\end{equation}
which holds for some $j>1$ if $2\alpha-1\not\in\Zset_{\le 0}$,
 is required as well as $\nu\neq 0$ for the former.
If $\alpha\in\Qset$, then $\alpha\pm\nu\not\in\Qset$
 and that if $\alpha-\nu$ or $\alpha+\nu\in\Qset$, then $\alpha\not\in\Qset$,
 since $\alpha$ or $\nu\not\in\Qset$.
Hence, if $2\alpha-1\notin\Zset_{\le 0}$, then one can take $j>1$
 for which conditions \eqref{eqn:thmfH1a} and \eqref{eqn:thmfH1b} hold simultaneously.

We next assume that $\alpha=\nu$ or $-\nu$ and $\alpha,\nu\neq 0$. %and that $-\alpha\notin\Nset$.
By \eqref{eqn:kappafH}
\begin{align*}
&
\kappa_{1\n}(x_3)=\alpha\neq 0,\quad
\kappa_{2j-1,\n}(x_3)=(2j-1)!(-s)^{j-1}\alpha\neq 0,\\
&
\kappa_{1\d}(x_3)=x_3\mp 1,\quad
\kappa_{2j-1,\d}(x_3)=\kappa_{1\d}(x_3)(x_3+1)^{j-1}(x_3-1)^{j-1},
\end{align*}
from which $n_1=1$, $\xi_{11}=\pm 1$, $a_{11}=j-1$,
 $n_{2j-1}=1$, $\xi_{2j-1,1}=\mp 1$ and $a_{2j-1,1}=j-1$,
 where the upper and lower signs are taken for $\alpha=\nu$ and $-\nu$, respectively.
So we see that condition~(i) of Proposition~\ref{prop:3b} holds for $j=2$.
Thus we obtain the desired result for $\mu\neq 0$.

We turn to the case of $\mu=0$.
Let $\mu=0$ and let $\alpha\not\in\Qset$ or $\nu\neq 0$.
From \eqref{eqn:kappafH} we easily see
 that $\deg(\kappa_{1\d})>\deg(\kappa_{1\n})$ and compute
\[
\Omega(x_3)=\exp\left(\int\frac{\alpha x_3+\nu}{x_3^2}\d x_3\right)
 =x_3^\alpha e^{-\nu/x_3},
\]
so that condition~(H1) holds.
We check the hypotheses of Proposition~\ref{prop:3b} for $\alpha,\nu\neq 0$.

Let $k=3$.
Assume that $\alpha,\nu\neq 0$.
Then by \eqref{eqn:kappafH}
\begin{align*}
&
\kappa_{1\n}(x_3)=\alpha x_3+\nu\not\equiv 0,\quad
\kappa_{3\n}(x_3)=-6s(\alpha x_3+\nu),\\
&
\kappa_{1\d}(x_3)=x_3^2,\quad
\kappa_{3\d}(x_3)=\kappa_{1\d}(x_3)x_3^2,
\end{align*}
from which $n_1=1$, $\xi_{11}=0$, $a_{11}=2$ and $n_3=0$.
We compute \eqref{eqn:bkappa} as
\[
\bar{\kappa}_{3b}(\xi)=((2(\alpha-1)-b+1)x_3+\nu)x_3,
\]
where $b\in\Nset$,
 so that the zero $x_3=0$ is simple for $\bar{\kappa}_{3b}(\xi)$ with any $b\in\Nset$.
 Eq.~\eqref{eqn:rhok} becomes
\[
\rho_{2j-1}(x_3)=2((\alpha-1)x_3+\nu).
\]
We see that
 $\bar{n}=0$, $\deg(\kappa_{1\d})+n_3=\deg(\rho_3)+1=2$ and 
\[
\bar{\rho}_3(x_3)=\frac{-6s\alpha}{2(\alpha-1)},\quad
\tilde{\rho}_3(x_3)
 =\frac{6s\nu}{\alpha-1}\not\equiv 0
\]
if $\alpha\neq 1$,
 and that $\bar{n}=1$ and
\[
\deg(\kappa_{1\d})+n_3=2>\max(\deg(\kappa_{3\n}),\deg(\rho_3))=1
\]
if $\alpha=1$.
So condition~(iv) or (v) of Proposition~\ref{prop:3b} holds,
 depending on whether $\alpha\neq 1$ or not,
 where condition \eqref{eqn:thmfH1b} is required for the former.
Thus, we complete the proof.
\end{proof}

%%%%%%%%%%%%%%%%%%%% Section of Hopf-Hopf will appear in the file sectionHH.tex

%%%%%%%%%%%%%%%%%%%%%%%%%
\section{Proofs of Theorems~\ref{thm:dH1} and \ref{thm:dH2}}

We now turn to the unfolding \eqref{eqn:dH} of double-Hopf bifurcations
 and reduce the problem to \eqref{eqn:dHp0} based on Corollary~\ref{cor:dH},
 as in Section~4.
We set $(\xi,\eta)=(r_2,r_1)$ or $(r_1,r_2)$
 and apply Theorem~\ref{thm:c} to \eqref{eqn:dHp0}
 with assistance of Proposition~\ref{prop:3b} in a similar way
 as in the proof of Theorem~\ref{thm:fH1}.
Eq.~\eqref{eqn:fol} becomes
\begin{equation}\label{eqn:dH2}
\frac{\d r_1}{\d r_2}=\frac{r_1(s r_1^2+\alpha r_2^2+\nu)}{r_2(\beta r_1^2-r_2^2+\mu)}
\end{equation}
and
\begin{equation}\label{eqn:dH1}
\frac{\d r_2}{\d r_1}=\frac{r_2(\beta r_1^2-r_2^2+\mu)}{r_1(s r_1^2+\alpha r_2^2+\nu)}
\end{equation}
for $(\xi,\eta)=(r_2,r_1)$ and $(r_1,r_2)$, respectively.
Recall that $s=\pm 1$.

\begin{proof}[Proof of Theorem~$\ref{thm:dH1}$]
We consider \eqref{eqn:dH2}
 and take $r_1=0$ as the integral curve, i.e., $\varphi(r_2)=0$.
We compute \eqref{eqn:kappa} as
\begin{equation}\label{eqn:kappadH2}
\begin{split}
&
\kappa_1(r_2)=-\frac{\alpha r_2^2+\nu}{r_2(r_2^2-\mu)},\quad
\kappa_{2j}(r_2)=0,\\
&
\kappa_{2j+1}(r_2)=-(2j+1)!\beta^{j-1}
 \frac{(\alpha\beta+s)r_2^2+\beta\nu-\mu s}{r_2(r_2^2-\mu)^{j+1}},\quad
j\in\Nset.
\end{split}
\end{equation}

We begin with the case of $\mu\neq 0$.
Additionally, let $\alpha+\nu/\mu\not\in\Qset$ or $\nu/\mu\not\in\Qset$.
Replacing $r_1$, $r_2$ and $\nu$
 with $\sqrt{\mu}\,r_1$, $\sqrt{\mu}\,r_2$ and $\mu\nu$, respectively,
 we take $\mu=1$ and have $\alpha+\nu\not\in\Qset$ or $\nu\not\in\Qset$.
We easily see by \eqref{eqn:kappadH2} that $\deg(\kappa_{1\d})>\deg(\kappa_{1\n})$
 and compute
\[
\Omega(r_2)=\exp\left(\int-\frac{\alpha r_2^2+\nu}{r_2(r_2^2-1)}\d r_2\right)
 =r_2^\nu(r_2^2-1)^{-(\alpha+\nu)/2},
\]
so that condition~(H1) holds.
If $\nu\not\in\Qset$, $\alpha\not\in\Zset_{\ge 0}$
 and $\beta\nu-s,(\alpha+\nu)s-(\beta\nu-s)\neq 0$,
 then $\alpha+\nu,\beta\neq 0$;
 $\alpha+\nu=0$;
 or $\beta=0$ and $\alpha+\nu\neq 0$,
 as well as $\alpha,\nu\neq 0$.
On the other hand, if $\alpha+\nu\not\in\Qset$, $\alpha\not\in\Zset_{\ge 0}$
 and $\beta\nu-s,(\alpha+\nu)s-(\beta\nu-s)\neq 0$,
 then $\nu,\beta\neq 0$; $\nu=0$ and $\alpha\neq -1$; or $\beta=0$,
 as well as $\alpha,\alpha+\nu\neq 0$.
In the following,
 we check the hypotheses of Proposition~\ref{prop:3b}
 for $\alpha+\nu,\alpha,\nu,\beta,\beta\nu-s\neq 0$,
 for $\alpha+\nu=0$ and $\nu,\beta\nu-s\neq 0$,
 for $\nu=0$ and $\alpha\neq 0$,
 and for $\beta=0$ and $\alpha+\nu,\alpha,\nu\neq 0$.
Here $\beta$ may be zero in the second and third cases,
 and $\beta\nu-s\neq 0$ also holds in the latter two cases.

Let $k=2j+1$ for $j\ge 1$.
We first assume that $\alpha+\nu,\alpha,\nu,\beta,\beta\nu-s\neq 0$.
Then by \eqref{eqn:kappadH2}
\begin{align*}
&
\kappa_{1\n}(r_2)=-(\alpha r_2^2+\nu),\quad
\kappa_{2j+1,\n}(r_2)
 =-(2j+1)!\beta^{j-1}((\alpha\beta+s)r_2^2+\beta\nu-s),\\
&
\kappa_{1\d}(r_2)=r_2(r_2+1)(r_2-1),\quad
\kappa_{2j+1,\d}(r_2)=\kappa_{1\d}(r_2)(r_2+1)^j(r_2-1)^j,
\end{align*}
from which $n_1=2$, $\xi_{11}=-1$, $\xi_{2j+1,1}=-1$,  $a_{11}=a_{12}=j$ and $n_{2j+1}=0$,
 since $(\alpha\beta+s)+(\beta\nu-s)=(\alpha+\nu)\beta\neq 0$.
We also compute \eqref{eqn:bkappa} as
\[
\bar{\kappa}_{2j+1,b}(r_2)
 =-(r_2+1)(r_2-1)((2j(\alpha+1)+b_1+b_2-2)r_2^2-(b_1-b_2)r_2+2j\nu)
\]
with $b=(b_1,b_2)\in\Nset$.
If and only if
\begin{equation}
\alpha+\nu\neq\frac{1-b_1}{j}-1\quad\left(\mbox{resp.}\quad
\frac{1-b_2}{j}-1\right),
\label{eqn:thmdH1a}
\end{equation}
then the zero $r_2=-1$ (resp. $r_2=1$) is simple for $\bar{\kappa}_{2j+1,b}(r_2)$.
Hence, if $\alpha+\nu=-1$,
 then the zeros $r_2=\pm 1$ of $\bar{\kappa}_{2j+1}(r_2)$ are double,
 so that condition~(ii) of Proposition~\ref{prop:3b} holds.
Note that if $\alpha+\nu=-1$ for $\nu\not\in\Qset$,
 then $\alpha\not\in\Qset\supset\Zset_{\ge 0}$
 and $(\alpha+\nu)s-(\beta\nu-s)=-\beta\nu\neq 0$.

Suppose that $\alpha+\nu\neq -1$.
If $\alpha+\nu>-1$ or $\alpha+\nu+2\not\in\Zset_{\le 0}$,
 then there exists an integer $j>0$
 such that both conditions in \eqref{eqn:thmdH1a} hold, i.e.,
 the zeros $r_2=\pm 1$ of $\bar{\kappa}_{2j+1,b}(r_2)$ are simple, for any $b\in\Nset^2$.
Eq.~\eqref{eqn:rhok} becomes
\[
\rho_{2j+1}(r_2)=-2j((\alpha+1)r_2^2+\nu).
\]
If $\alpha\neq -1$,
 then $\bar{n}=0$, $\deg(\kappa_{1\d})+n_{2j+1}=\deg(\rho_{2j+1})+1=3$ and
\begin{align*}
&
\bar{\rho}_{2j+1}(r_2)
 =\frac{(2j+1)!\,\beta^{j-1}(\alpha\beta+s)}{2j(\alpha+1)},\\
&
\tilde{\rho}_{2j+1}(r_2)
 =\frac{(2j+1)!\,\beta^{j-1}((\alpha+\nu)s-(\beta\nu-s))}{\alpha+1}.
\end{align*}
If $\alpha=-1$ and $(\alpha+\nu)s-(\beta\nu-s)\neq 0$, i.e., $\beta\neq s$,
 then $\bar{n}=2$ and
\[
\deg(\kappa_{1\d})+n_{2j+1}=3>\max(\deg(\kappa_{2j+1,\n}),\deg(\rho_{2j+1})+1)=2.
\]
Noting that $(\alpha+\nu)s-(\beta\nu-s)=-(\alpha+\nu)(\alpha+1)\beta$ when $\alpha\beta+s=0$,
 we see that condition~(iv) or (v) of Proposition~\ref{prop:3b} holds
 if $(\alpha+\nu)s-(\beta\nu-s)\neq 0$,
 depending on whether $\alpha\neq -1$ or not, where the condition
\begin{equation}
-\rho_{2j+1,0}=2j(\alpha+1)\notin\Nset,
\label{eqn:thmdH1b}
\end{equation}
which holds for some $j\in\Nset$ if $2(\alpha+1)\notin\Nset$, is required for the former.
Note that if $2(\alpha+1)\notin\Nset$, then 
 conditions \eqref{eqn:thmdH1a} and \eqref{eqn:thmdH1b} hold simultaneously
 for some $j\in\Nset$.
Moreover, if $2(\alpha+1)$ is a positive odd number,
 then $\alpha\notin\Zset_{\ge 0}$ and condition~(iii) holds for $k=2j+1$
 when $j>0$ is an odd number.
Actually, if $2j(\alpha+1)=2\ell-1$, $\ell\in\Nset$,
 and Eq.~\eqref{eqn:prop3b} has a polynomial solution,
 then it has the form
\[
z=\sum_{i=1}^\ell z_ir_2^{2i-1},\quad z_1,\ldots,z_\ell\in\Cset,
\]
but never satisfies \eqref{eqn:prop3b}
 since the left hand side of \eqref{eqn:prop3b} has no even-order monomial.
Thus, under our present assumptions,
 condition~(H2) holds if $\alpha\not\in\Zset_{\ge 0}$,
 $\alpha+\nu+2\not\in\Zset_{\le 0}$  and $(\alpha+\nu)s-(\beta\nu-s)\neq 0$.

We next assume that $\nu,\beta\nu-s\neq 0$ but $\alpha+\nu=0$.
Then $\alpha\neq 0$.
By \eqref{eqn:kappadH2}
\begin{align*}
&
\kappa_{1\n}(r_2)=-\alpha,\quad
\kappa_{2j+1,\n}(r_2)=(2j+1)!\beta^{j-1}(\beta\nu-s),\\
&
\kappa_{1\d}(r_2)=r_2,\quad
\kappa_{2j+1,\d}(r_2)=\kappa_{1\d}(r_2)(r_2+1)^j(r_2-1)^j,
\end{align*}
from which $n_{2j+1}=2$, $\xi_{2j+1,1}=-1$, $\xi_{2j+1,1}=1$, $a_{2j+1,1}=a_{2j+1,2}=j$
 and $n_1=0$.
Since $\beta\nu-s\neq 0$,
 we have $\kappa_{2j+1,\n}(r_2)\not\equiv 0$
 and condition~(i) of Proposition~\ref{prop:3b} holds for $j=1$ even if $\beta=0$.

We next assume that $\alpha\neq 0$ but $\nu=0$.
Then $(\alpha+\nu)s-(\beta\nu-s)=(\alpha+1)s$.
If $\beta\neq 0$, then by \eqref{eqn:kappadH2}
\begin{align*}
&
\kappa_{1\n}(r_2)=-\alpha r_2,\quad
\kappa_{2j+1,\n}(r_2)=-(2j+1)!\beta^{j-1}((\alpha\beta+s)r_2^2-s),\\
&
\kappa_{1\d}(r_2)=(r_2+1)(r_2-1),\quad
\kappa_{2j+1,\d}(r_2)=\kappa_{1\d}(r_2)r_2(r_2+1)^j(r_2-1)^j,
\end{align*}
from which $n_1=2$, $\xi_{11}=-1$, $\xi_{12}=1$, $a_{11}=a_{12}=j$,
 $n_{2j+1}=1$, $\xi_{2j+1,1}=0$ and $a_{2j+1,1}=1$,
 so that condition~(i) of Proposition~\ref{prop:3b} holds.
If $\beta=0$, then
\begin{align*}
&
\kappa_{1\n}(r_2)=-\alpha r_2,\quad
\kappa_{3\n}(r_2)=-6s,\\
&
\kappa_{1\d}(r_2)=(r_2+1)(r_2-1),\quad
\kappa_{3\d}(r_2)=\kappa_{1\d}(r_2)r_2,
\end{align*}
from which $n_3=1$, $\xi_{31}=0$, $a_{31}=1$ and $n_1=0$,
 so that condition~(i) of Proposition~\ref{prop:3b} holds.
Note that $\kappa_{2j+1,\n}(r_2)\equiv 0$ for $j>1$ when $\beta=0$.

We finally assume that $\alpha+\nu,\alpha,\nu\neq 0$ but $\beta=0$.
By \eqref{eqn:kappadH2}
\begin{align*}
&
\kappa_{1\n}(r_2)=-(\alpha r_2^2+\nu),\quad
\kappa_{3\n}(r_2)=-6s,\\
&
\kappa_{1\d}(r_2)=r_2(r_2+1)(r_2-1),\quad
\kappa_{3\d}(r_2)=\kappa_{1\d}(r_2),
\end{align*}
from which $n_1=n_3=0$.
Eq.~\eqref{eqn:rhok} becomes
\[
\rho_3(r_2)=-2(\alpha r_2^2+\nu).
\]
If $2\alpha\notin\Nset$, then $\bar{n}=0$, $\bar{\rho}_3(r_2)\equiv 0$,
 $\tilde{\rho}_3(r_2)\equiv-6s\neq 0$
 and $-\rho_{30}=2\alpha\notin\Nset$,
 so that condition~(iv) of Proposition~\ref{prop:3b} holds.
Note that $\deg(\kappa_{1\d})+n_3=\deg(\rho_3)+1=3$.
Moreover, if $2\alpha$ is a positive odd number,
 then $\alpha\notin\Zset_{\ge 0}$ and condition~(iii) of Proposition~\ref{prop:3b}
 holds for  $k=3$,
 as in the above argument for the first case with $\alpha+\nu\neq -1$.
Thus, we obtain the desired result for $\mu\neq 0$.

We turn to the case of $\mu=0$.
Let $\mu=0$ and let $\alpha\not\in\Qset$ or $\nu\neq 0$.
We easily see by \eqref{eqn:kappadH2} that $\deg(\kappa_{1\d})>\deg(\kappa_{1\n})$
 and compute
\[
\Omega(r_2)=\exp\left(\int\frac{\alpha r_2^2+\nu}{r_2^3}\d r_2\right)
 =r_2^{\alpha}e^{-\nu/2r_2^2},
\]
so that condition~(H1) holds.
We check the hypotheses of Proposition~\ref{prop:3b}
 for $\alpha,\beta,\nu\neq 0$ and for $\beta=0$ and $\alpha,\nu\neq 0$.

Let $k=3$.
Assume that $\alpha,\beta,\nu\neq 0$.
Then by \eqref{eqn:kappadH2}
\begin{align*}
&
\kappa_{1\n}(r_2)=-(\alpha r_2^2+\nu),\quad
\kappa_{3\n}(r_2)
 =-6((\alpha\beta+s)r_2^2+\beta\nu),\\
&
\kappa_{1\d}(r_2)=r_2^3,\quad
\kappa_{3\d}(r_2)=\kappa_{1\d}(r_2)r_2^2,
\end{align*}
from which $n_1=1$, $\xi_{11}=0$, $a_{11}=2$ and $n_3=0$.
Eqs.~\eqref{eqn:bkappa} and \eqref{eqn:rhok} become
\[
\bar{\kappa}_{3b}(r_2)=-r_2((2(\alpha+1)+b-1)r_2^2+2\nu),\quad
\rho_3(r_2)=-2((\alpha+1)r_2^2+\nu),
\]
where $b\in\Nset$.
The zero $r_2=0$ of $\bar{\kappa}_{3b}(r_2)$ is simple for any $b\in\Nset$.
Suppose that $\alpha\neq -1$.
Then $\deg(\kappa_{1\d})+n_3=\deg(\rho_3)+1=3$.
If $\alpha\beta+s\neq 0$,
 then $\bar{n}=0$ and
\[
\bar{\rho}_3(r_2)=\frac{3(\alpha\beta+s)}{(\alpha+1)},\quad
\tilde{\rho}_3(r_2)
 =\frac{6(\beta-s)}{\alpha+1},
\]
and if $\alpha\beta+s=0$,
 then $\bar{n}=0$, $\bar{\rho}_3(r_2)\equiv 0$
 and $\tilde{\rho}_3(r_2)\equiv-6\beta\nu\neq 0$.
On the other hand, suppose that $\alpha=-1$.
If $\alpha\beta+s=-(\beta-s)\neq 0$, then $\bar{n}=2$ and
\[
\deg(\kappa_{1\d})+n_{2j+1}=3>\max(\deg(\kappa_{2j+1,\n}),\deg(\rho_{2j+1})+1)=2.
\]
Thus, we see that condition~(iv) or (v) of Proposition~\ref{prop:3b} holds if $\beta\neq s$,
 depending on whether $\alpha\neq -1$ or not,
 where the condition $\alpha\not\in\Nset$,
 which follows from $-\rho_{30}=2(\alpha+1)$ as in the above argument,
 is required for the former.

We {finally} assume that $\alpha,\nu\neq 0$ but $\beta=0$.
By \eqref{eqn:kappadH2}
\begin{align*}
&
\kappa_{1\n}(r_2)=-(\alpha r_2^2+\nu),\quad
\kappa_{3\n}(r_2)=-6s,\\
&
\kappa_{1\d}(r_2)=r_2^3,\quad
\kappa_{3\d}(r_2)=\kappa_{1\d}(r_2),
\end{align*}
from which $n_1=n_{2j+1}=0$.
Eq.~\eqref{eqn:rhok} becomes
\[
\rho_3(r_2)=-2(\alpha r_2^2+\nu),
\]
so that $\bar{n}=0$, $\deg(\kappa_{1\d})+n_{2j+1}=\deg(\rho_{2j+1})+1=3$ and
\[
\bar{\rho}_3(r_2)\equiv 0,\quad
\tilde{\rho}_3(r_2)\equiv -6s\neq 0.
\]
Hence,
 then conditions~(iv) and (iii) of Proposition~\ref{prop:3b} holds
 if $-\rho_{30}=2\alpha\not\in\Nset$ and $2\alpha$ is an odd number, respectively.
 Thus, we complete the proof.
\end{proof}

\begin{proof}[Proof of Theorem~$\ref{thm:dH2}$]
We consider \eqref{eqn:dH1} and take $r_2=0$ as the integral curve,
 i.e., $\varphi(r_1)=0$.
Replacing $r_1$ with $\sqrt{-s}\,r_1$,
 we rewrite \eqref{eqn:dH1} as
\begin{equation}
\frac{\d r_2}{\d r_1}
 =\frac{r_2(-r_2^2-\beta sr_1^2+\mu)}{r_1(\alpha r_2^2-r_1^2+\nu)},
\label{eqn:dH1r}
\end{equation}
which has the form of \eqref{eqn:dH2} with $s=-1$.
Applying Theorem~\ref{thm:dH1} to \eqref{eqn:dH1r},
 we easily obtain the desired result.
\end{proof}

\section*{Acknowledgments}
The authors acknowledge support from Japan Society for the Promotion of Science (JSPS)
 Fellowship, which enables one of the authors (P.A.) to stay in Kyoto
 as a JSPS International Fellow (ID Number S17113)
 and to make collaboration with the other (K.Y.).
They thank Andrzej J. Maciejewski for his helpful comments,
 and one of the anonymous referees for pointing out errors in the original manuscript.
K.Y. also appreciates support by JSPS Kakenhi Grant Number JP17H02859.

\end{document}